\documentclass[final]{siamltex}

\usepackage{graphicx}
\usepackage{subfigure}
\usepackage{amsmath}
\usepackage{booktabs} 
\usepackage{array}
\usepackage{slashed}
\usepackage{amsmath,bm}
\usepackage{appendix}
\usepackage{mathrsfs,psfrag,eepic,epsfig}
\usepackage{fancyhdr}
\usepackage{graphicx}
\usepackage{makecell}
\usepackage{multirow}
\usepackage{float}
\usepackage{amsmath,amssymb}
\usepackage{algorithm}
\usepackage{algorithmic}
\usepackage{tikz}
\usepackage{epstopdf}
\usepackage{caption}
\usepackage[algo2e,ruled,vlined]{algorithm2e} 
\usepackage{threeparttable}
\usepackage{hyperref}
\hypersetup{hidelinks,
	colorlinks=true,
	allcolors=blue,
	pdfstartview=Fit,
	breaklinks=true}

\newtheorem{remark}{Remark}[section]

\title{A deep backward regression-based scheme for high-dimensional nonlinear partial differential equations}

\author{Qiang Han\thanks{School of Mathematics, Yangzhou University, Yangzhou 225002, PR China. Email: {\tt hanqiang@yzu.edu.cn}.}
       \and Shaolin Ji\thanks{Zhongtai Securities Institute for Financial Studies, Shandong University, Jinan, Shandong 250100, PR China. Email: {\tt jsl@sdu.edu.cn}.}
\and Yunzhang Li\thanks{Research Institute of Intelligent Complex Systems, Fudan University, Shanghai 200433, China. Email: {\tt li\_yunzhang@fudan.edu.cn}.
}}

\begin{document}

\maketitle

\begin{abstract}
We propose a deep backward regression-based (DBR) scheme for solving
high-dimensional nonlinear parabolic partial differential equations. Building on
the DBDP method of Hur\'e, Pham, and Warin~\cite{HCPHWX20}, the proposed method
reformulates the local backward losses through conditional expectations and trains
the resulting regression problems sequentially in time. This conditional-expectation
formulation replaces pathwise Brownian fluctuations in the Euler residual by their
averaged effect and therefore provides an intrinsic variance-reduction mechanism
before loss evaluation. In practice, the conditional expectations are approximated
by local multi-path Monte Carlo averages, which leads to smoother training targets
and improved numerical stability. Numerical experiments show that DBR performs
competitively on standard high-dimensional benchmarks and is more stable than
DBDP1 on the challenging unbounded benchmark considered in Example~2. Under an
idealized population-loss minimization setting, we provide an error analysis and
establish a half-order convergence result under suitable approximation and
integrability assumptions. We also discuss an extension to variational inequalities.
\end{abstract}



\section{Introduction}
In this paper, we develop a probabilistic deep learning scheme for approximating
solutions of high-dimensional quasilinear parabolic PDEs of the form
\begin{equation}
\begin{aligned}
\begin{cases}
\partial_t u+\mu \cdot D_x u+\frac{1}{2} \operatorname{Tr}\left(\sigma\sigma^{\top} D_x^2 u\right)+f\left(\cdot, \cdot, u, \sigma^{\top} D_x u\right)=0 & \text { on }[0, T) \times \mathbb{R}^d, \\
u(T, \cdot)=g(\cdot)& \text { on } \mathbb{R}^d,
\end{cases}
\end{aligned}
\label{PDE}
\end{equation}
where $\mu$ is a map from $[0,T]\times\mathbb{R}^{d}\rightarrow \mathbb{R}^{d}$ and $\sigma$ is a map from $[0,T]\times\mathbb{R}^{d}\rightarrow \mathbb{R}^{d\times d}$;
$f:[0,T]\times\mathbb{R}^{d}\times\mathbb{R}^{n}\times\mathbb{R}^{n\times d}\rightarrow\mathbb{R}^{n}$
and $g:\mathbb{R}^{d}\rightarrow \mathbb{R}^{n}$  represent the nonlinear generator and the terminal function respectively;
the operators $D_x$ and $D_x^2$ refer to the first and second order spatial derivatives;
the symbol $\cdot$ represents the scalar product and $^\top$ denotes the transpose of vector or matrix.

The pioneer works of  Pardoux and Peng~\cite{PEPS92} and Peng~\cite{PS91} show that the quasilinear parabolic PDEs (\ref{PDE}) are associated with
Markovian forward backward stochastic differential equations (FBSDEs) due to the nonlinear Feynman-Kac formula.
Since then, significant interest has been directed toward the probabilistic interpretation of PDE solutions.
By these interpretations, researchers have developed probabilistic numerical methods to approximate PDE solutions
(see \cite{DFMS06,DFMS08,FATNWX11} and references therein). Concurrently,
the design of algorithms to solve FBSDEs has been intensively studied.
A significant milestone in this classical numerical analysis line was established by Zhang in his seminal work \cite{ZJ04}.   This work provided a robust theoretical foundation.
In fact, our framework is fundamentally rooted in the classical backward stochastic differential equations (BSDEs) numerical analysis, drawing particular inspiration from Zhang \cite{ZJ04}, which proposed the idea of solving BSDEs through backward induction on conditional expectations and proved  the rate of convergence in the strong $L^2$ sense. The DBR method effectively extends these fundamental principles from traditional numerical method into the deep learning paradigm, aiming to handle the curse of dimensionality.
Up to now, the proposed numerical algorithms to solve nonlinear FBSDEs
have shown their remarkable performance with respect to the high-order problems.
For instance, the Euler schemes with convergence rates of $\frac{1}{2}$ or $1$ are presented in \cite{BBTN04,GELC07,ZWCLPS06}.
Multi-step schemes \cite{HQJS24,HQJS23,GETP16,ZWFYZT14}
achieve convergence rates greater than or equal to
$1$, although they require terminal conditions over multiple steps.
Furthermore, one-step schemes with convergence rates of at least $2$ have been proposed in \cite{CJFDC14,ZCWJZW19,HQLSZQ24}.

However, most of these numerical methods are  unsuitable for solving nonlinear PDEs in dimensions greater than 4 (see \cite{BHJGM04}).
To address this issue, some techniques such as the parallel computing method, the sparse grid method, and the variance reduction method (including
the control variate methods, the importance sampling methods, and the multilevel Monte Carlo methods et al.)
have proven effective for nonlinear PDEs  with the dimension below 20 (see \cite{GELSJGTPVC16}).
The high dimensional problem (that is the so-called curse of dimensionality)
remains a major challenge in the numerical solutions of PDEs, even in the broader field of scientific computing.

A major development occurred in 2017, when E, Han, and Jentzen introduced the deep
BSDE method, one of the first deep learning-based numerical methods for
high-dimensional nonlinear PDEs and BSDEs.
Since then, many neural-network-based numerical methods have been proposed.  They are designed by adjusting, reformulating or extending the deep BSDE method
to solve high dimensional PDEs (see \cite{EWHMJAKT19,GMPHWX22,HCPHWX20,HJHRLJ24,JSPSPYZX25,PHWXGM21,GMLMPHWX22,GMMJWX22} and many others)
and these deep numerical algorithms have also shown their remarkable performance on the high-dimensional problems.
Recently, Cai, Fang and Zhou \cite{ CWFSZT25,CWFSZT24}  introduce the SOC-MartNet method to solve ultra-high-dimensional quasilinear parabolic PDEs and it demonstrates the strong performance on large-scale benchmarks. Furthermore, they propose the deep random difference method in \cite{CWFSZT25a} to reduce variance and improve stability of the SOC-MartNet method.

At the forefront of deep learning methodologies for high-dimensional PDEs, the seminal work by Hure, Pham, and Warin \cite{HCPHWX20} proposed the DBDP1 method, which has become a leading paradigm in the current literature.
It estimates the solutions and their gradient  by the loss functions, for $i=N-1,N-2,\cdots,0$
\begin{equation}\label{DBDP-loss}
J_i^{HPW}(\theta_{i}) = \mathbb{E}\left[ \left|\mathcal{Y}_i - \mathcal{Y}_{i+1}^*
    -hf(t_i,X_i^\pi,\mathcal{Y}_i,\mathcal{Z}_i ) + \mathcal{Z}_i\Delta W_i
    \right|^2\right],
\end{equation}
at each time step $t_i$, where $\mathcal{Y}_i: \mathbb{R}^d \rightarrow \mathbb{R}^n$ and
$\mathcal{Z}_i: \mathbb{R}^d \rightarrow \mathbb{R}^{n\times d}$ denote the sequences of network functions. Then update $(\mathcal{Y}_i^*,\mathcal{Z}_i^*)$ as the solutions of the local
minimization problems.
Thus, we only need to consider the local gradient as updating the parameter of neural networks, which significantly reduces the difficulty of backpropagation and solves the training difficulties brought by long-sequence time dependencies.
The loss function (\ref{DBDP-loss}) is based on the Euler discretization of the BSDE. It measures the residual of the Euler scheme which implies that during the training process, one can directly assess the accuracy of the scheme by observing whether the loss value approaches zero.
But the random noise term $\mathcal{Z}_i\Delta W_i$ in (\ref{DBDP-loss})  may lead to fluctuations especially when $\mathcal{Z}_i$ is large
(such as high volatility or complex hedging requirements).
Naturally, it may cause oscillations in the gradients which lead to divergence.
Besides, the loss function $J_i^{HPW}(\theta_{i})$ is prone to overfitting the noise of the simulated paths, leading to poor generalization ability.

To address these problems, we propose a new deep backward regression-based (DBR) method to approximate the solutions and their gradient of high-dimensional PDEs by the loss functions
\begin{equation}\label{DBR-loss-y}
F_{y,i}(\theta_{y,i}) = \mathbb{E}\left[ \left|\mathcal{Y}_i - \mathbb{E}_i\left[\mathcal{Y}_{i+1}^*
    +hf(t_i,X_i^\pi,\mathcal{Y}_i,\mathcal{Z}_i^* ) \right]
    \right|^2\right],
\end{equation}
and
\begin{equation}\label{DBR-loss-z}
F_{z,i}(\theta_{z,i}) = \mathbb{E}\left[ \left|\mathcal{Z}_i - \mathbb{E}_i\left[\mathcal{Y}_{i+1}^*\frac{\Delta W_i^\top}{h} \right]
    \right|^2\right],
\end{equation}
where $\mathbb{E}_i\left[\cdot\right] = \mathbb{E}\left[\cdot\big|\mathcal{F}_{t_i}\right]$
is the conditional expectation. The above regression-based loss functions which utilize conditional expectations are different from those proposed in Remark 2.3 of Germain, Pham and Warin \cite{GMPHWX22}.
Unlike the DBDP method, which operates directly on the Euler-discretized residual susceptible to pathwise noise  $\Delta W_i$, our DBR method exploits the conditional expectation representation central to classical frameworks like Zhang \cite{ZJ04}. By replacing pathwise noise with its averaged effect, the DBR scheme achieves intrinsic variance reduction. This design choice aligns conceptually with the $L^2$-regularity established by Zhang \cite{ZJ04}, which provides the theoretical foundation for smoothing the martingale integrand $Z$ via step-process approximations.
The accuracy of the DBR scheme can also be tested when computing at each time step the infimum of loss function, which should be equal to zero as the DBDP scheme in \cite{HCPHWX20}.
By explicitly incorporating the conditional expectations, the DBR scheme effectively performs a ``denoising" step prior to loss evaluation. In terms of implementation, this is typically realized through multi-path Monte Carlo averaging. Its essence lies in replacing
the simulated backward stochastic difference equations with their conditional expectation forms and transforming a projection-based stochastic optimization problem into a smoother deterministic function-fitting task.

The DBR algorithm is built on the following ingredients:
(i) the FBSDE is discretized by an Euler-type scheme inspired by Zhang~\cite{ZJ04};
(ii) neural networks with inputs \((t_i,X_i^\pi)\) are used to approximate the
conditional expectations at each time step;
(iii) the \(Z\)-component is trained first and then fixed when training the
\(Y\)-component;
(iv) the conditional-expectation-based losses \eqref{DBR-loss-y} and
\eqref{DBR-loss-z} are used for network training.

Moreover, the core of the DBR method is to approximate the conditional expectation through multi-path Monte Carlo averaging. This design shows some advantages when dealing with high-dimensional or complex-structured PDEs:
(i) noise smoothing: multi-path averaging mitigates the impact of single-path Brownian noise $\Delta W_i$, reducing numerical fluctuations (e.g., the loss function of regression-based methods does not directly depend on $\Delta W_i$).
(ii) generalization ability: the DBR method learns local solutions of multi-paths rather than local solutions of single paths, making them more likely to avoid overfitting.  In particular, for the unbounded benchmark with a complex structure considered in
Example~2, the DBR method mitigates the adverse effects of increasing dimensionality
more effectively than DBDP1 when \(d=15\) and \(d=20\); see Tables~5.5 and~5.6.
The method's consistent performance at these two dimensions, compared with the
substantial deterioration of DBDP1 on the same benchmark, provides evidence of its
practical utility for this class of challenging high-dimensional problems.

The remainder of this paper is organized as follows.
In Section 2, some essential materials  that are used in the whole paper are provided.
The DBR method is designed in Subsection 3.1. An error analysis for the proposed
DBR method is presented in Subsection 3.2 under the population-loss convention. In
addition, under suitable approximation and integrability assumptions, we establish a
half-order convergence result for the proposed DBR method in Subsection 3.3.
In Section 4, we extend the DBR scheme to variational inequalities and prove its corresponding convergence.
We report the numerical results in Section 5.
In the end, we compile the conclusion of this paper in Section 6.

\section{Preliminaries}
This section collects the notation and preliminary results used throughout the
paper, including the probabilistic representation of PDE solutions, neural-network
function classes, and the Euler discretization of FBSDEs.

\subsection{Deep neural networks as function approximators}

The universal approximation theorem in \cite{HKSMWH89} indicates that we can use the neural networks to approximate the unknown functions and this is also reasonable.
Therefore, we recall some notations and basic definitions with respect to the feedforward neural networks
which are a type of deep neural networks and  will be applied in the following parts.
Let
$$
\mathcal{LL}_{d_1, d_2}^\varrho=\left\{\psi: \mathbb{R}^{d_1} \rightarrow \mathbb{R}^{d_2}: \exists~(\mathcal{W}, \beta) \in \mathbb{R}^{d_2 \times d_1} \times \mathbb{R}^{d_2}, \psi(x)=\varrho(\mathcal{W} x+\beta)\right\}
$$
represent the set of the layer functions with the input dimension $d_1$, the output dimension $d_2$
and the nonlinear function (called the activation function) $\varrho:\mathbb{R}^{d_2} \rightarrow \mathbb{R}^{d_2}$. The activation function $\varrho$  works by utilizing componentwise the one
dimensional activation function, namely $\varrho(x_1,x_2,\ldots,x_{d_2}) = (\overline{\varrho}(x_1),\overline{\varrho}(x_2),\ldots,\overline{\varrho}(x_{d_2}))$,
with $\overline{\varrho}(x):\mathbb{R}\rightarrow \mathbb{R}$, to the affine map   $x\in \mathbb{R}^{d_1}\rightarrow\mathcal{W} x+\beta\in \mathbb{R}^{d_2}$ with the weight $\mathcal{W}\in \mathbb{R}^{d_2 \times d_1}$ and the bias $\beta \in \mathbb{R}^{d_2}$. We write $\mathcal{LL}_{d_1, d_2}^\varrho$ as $\mathcal{LL}_{d_1, d_2}$
when $\varrho$ denotes the identity function.
Thus, the set of the feedforward neural networks can be
defined as
$$
\begin{aligned}
\mathcal{NN}_{\mathbf{d}_0, \mathbf{d}}^{\varrho, L, \mathbf{n}}= & \left\{\varphi: \mathbb{R}^{\mathbf{d}_0} \rightarrow \mathbb{R}^{\mathbf{d}}: \exists~ \psi_0 \in \mathcal{LL}_{\mathbf{d}_0, n_0}^{\varrho_0}, \exists~ \psi_\ell \in \mathcal{LL}_{n_{\ell-1}, n_{\ell}}^{\varrho_{\ell}}, \ell=1,2 \ldots,L-1,\right. \\
& \left.\exists~ \psi_{L} \in \mathcal{LL}_{n_{L-1}, \mathbf{d}}, \varphi=\psi_{L} \circ \psi_{L-1} \circ \cdots \circ\psi_0\right\}
\end{aligned}
$$
with the input dimension $\mathbf{d}_0$, the output dimension $\mathbf{d}$ and $L$ hidden layers with $n_\ell$ neurons in the $\ell$-th layer for $\ell = 0,1,\ldots,L-1$. Naturally,
the coefficients $\mathbf{d}_0, \mathbf{d}, L, \mathbf{n} = \{n_\ell\}_{0\le\ell\le L-1}$
and the activation $\varrho = \{\varrho_\ell\}_{0\le\ell\le L-1}$ form the architecture of the network.

\subsection{Probabilistic representation for the PDE solutions}
Our deep probabilistic numerical method for the PDEs (\ref{PDE}) relies on FBSDEs representation
of its solutions. The well-known nonlinear Feynman-Kac formula (see \cite{PEPS90,PS91,PEPS92}) shows that
on the filtered complete probability space $(\Omega,\mathcal{F},(\mathcal{F}_t)_{0\le t\le T},\mathbb{P})$,
the solution $(Y,Z)\in \mathbb{R}^n\times \mathbb{R}^{n\times d}$ of the FBSDEs
\begin{equation}
\left\{
\begin{aligned}
dX_{t}= &\mu(t,X_{t})dt+\sigma(t,X_{t})dW_{t},\quad X_{0}= x_0,\\
-dY_{t}= &f(t,X_{t},Y_{t},Z_{t})dt-Z_{t}dW_{t},\quad Y_{T}= g(X_{T}),\\
\end{aligned}
\right.  \label{dFBSDE}%
\end{equation}
related to the solution $u$ of the PDEs (\ref{PDE}) via
\begin{equation}
Y_{t} = u(t,X_{t}),\qquad Z_{t} = \sigma^\top(t,X_{t}) D_xu(t,X_{t}),
\label{FK}%
\end{equation}
where \(u\) is a sufficiently smooth classical solution. In order to include
unbounded terminal functions and unbounded solutions, we shall use the following
polynomial growth class instead of the bounded class \(C_b^{1,2}\). We write
\[
u\in C_p^{1,2}([0,T]\times\mathbb{R}^d)
\]
if \(u\in C^{1,2}([0,T]\times\mathbb{R}^d)\) and there exist constants
\(C_u>0\) and \(q_u\geq 0\) such that
\begin{equation}
\begin{aligned}
|u(t,x)|+|D_xu(t,x)|+|D_x^2u(t,x)|
\leq C_u\left(1+|x|^{q_u}\right),
\qquad (t,x)\in[0,T]\times\mathbb{R}^d .
\end{aligned}
\label{poly-growth-u}
\end{equation}
The bounded case \(C_b^{1,2}\) corresponds to \(q_u=0\). Under this regularity and
the integrability assumptions stated below, the nonlinear Feynman-Kac formula yields
\[
Y_t=u(t,X_t),
\qquad
Z_t=\sigma^\top(t,X_t)D_xu(t,X_t).
\]
Here \((\mathcal F_t)_{0\le t\le T}\) denotes the standard Brownian filtration;
\(W_t\) denotes a \(d\)-dimensional Brownian motion; \(x_0\in\mathbb R^d\)
represents a given initial value of the stochastic differential equation (SDE)
in \eqref{dFBSDE}.

For readers' convenience, before giving the time-discretization scheme of FBSDEs (\ref{dFBSDE}),
we introduce the following notations.
We denote the grid of the time interval $[0,T]$ by $\pi$, namely
$\pi=\{0=t_0<t_1< \cdots< t_N=T\}$, where $t_i=ih, h=\frac{T}{N}$
for $i=0,1,\cdots,N,N\in \mathbb{N}^+$
and $\Delta W_{i}=W_{t_{i+1}}-W_{t_i}$ the $(i+1)$-th Brownian motion increment.
Then the Euler time-discretization of the backward stochastic differential equation (BSDE)  in (\ref{dFBSDE}), at the mesh points $t_i$, is
 as follows
\begin{equation}
Y_{i}^{\pi}= Y_{i+1}^{\pi} + hf(t_{i},X_i^\pi,Y_{i}^\pi,Z_i^\pi) - Z_i^\pi\Delta W_{i},
\end{equation}
which also reads as the conditional expectation formula (see \cite{ZJ04})
\begin{equation}
\left\{
\begin{aligned}
Y_{i}^{\pi}= & \mathbb{E}_{i}\Big[Y_{i+1}^{\pi}+hf(t_{i},X_i^\pi,Y_{i}^\pi,Z_i^\pi)\Big],\\
Z_{i}^{\pi} = &\mathbb{E}_{i}\big[Y_{i+1}^{\pi}\frac{\Delta W_{i}^\top}{h}\big],
\end{aligned}
\right.
\label{NS}%
\end{equation}
where
$ X_{i+1}^{\pi}=  X_{i}^{\pi} + h \mu(t_{i},X_{i}^\pi) +\sigma(t_{i},X_{i}^\pi)\Delta W_{i}$; $X_{i}^{\pi}$, $Y_{i}^{\pi}$ and $Z_{i}^{\pi}$ denote the time-discretization form of $X$, $Y$ and $Z$ at $t_{i}$ respectively.
\medskip
\noindent
Since the first equation in the above conditional expectation formula is implicit
with respect to \(Y_i^\pi\), we impose a standard small time-step condition to ensure
its well-posedness. Let \(L_y\) denote a Lipschitz constant of \(f\) with respect to
the \(y\)-variable. Under Assumption (iii) below, one can take \(L_y\leq L\). We
assume throughout that
\begin{equation}
hL_y<1.
\label{time-step-condition}
\end{equation}
Indeed, for fixed \(i\), \(X_i^\pi\), and \(Z_i^\pi\), the map
\[
y\longmapsto
\mathbb{E}_i\left[Y_{i+1}^\pi\right]
+
h f(t_i,X_i^\pi,y,Z_i^\pi)
\]
is a contraction on \(\mathbb{R}^n\) whenever \eqref{time-step-condition} holds.
Hence the implicit equation for \(Y_i^\pi\) admits a unique solution at each time
step.
\medskip

\section{Deep backward regression-based method}

In this section, we formulate the DBR method to approximate the conditional expectations in scheme (\ref{NS}) using learned data.
Our aim is to provide a fully implementable algorithm.
Subsequently, we present a comprehensive error analysis of the DBR method. Finally, we investigate its convergence rate.

\subsection{Deep learning method to approximate the conditional expectation}

To implement the time-discrete scheme \eqref{NS}, the conditional expectations must
be approximated numerically. We approximate these conditional expectations by
training neural networks at each time step.

In what follows, the forward process $X_{\cdot}^{\pi,m} $ is simulated by Monte Carlo simulation as follows, for $i=0,1,\cdots,N-1$
\begin{equation}\label{SDE-M}
X_{i+1}^{\pi,m} =X_{i}^{\pi,m} +\mu(t_i,X_i^{\pi,m})h+\sigma(t_i,X_i^{\pi,m})\Delta W_i^m,  \quad X_{0}^{\pi,m}= x_0,
\end{equation}
where $\Delta W_i^m \sim \mathcal{N}(0, h)$,
 $m = 1,2,\cdots,M$, $M\in\mathbb{N}$.
The backward process $Y_{\cdot}^{\pi,m} $ is simulated by Monte Carlo method in the following ways, for $i=N-1,N-2,\cdots,0$
\begin{equation}\label{BSDE-pi-m}
Y_{i}^{\pi,m}= Y_{i+1}^{\pi,m} + hf(t_{i},X_i^{\pi,m},Y_{i}^{\pi,m},Z_i^{\pi,m}) - Z_i^{\pi,m}\Delta W_{i}^m,
\end{equation}
which also expresses as the conditional expectation form
\begin{equation}
\left\{
\begin{aligned}
Y_{i}^{\pi,m}= & \mathbb{E}_{i}\Big[Y_{i+1}^{\pi,m}+hf(t_{i},X_i^{\pi,m},Y_{i}^{\pi,m},Z_i^{\pi,m})\Big],\\
Z_{i}^{\pi,m} = &\mathbb{E}_{i}\big[Y_{i+1}^{\pi,m}\frac{(\Delta W_{i}^m)^\top}{h}\big].
\end{aligned}
\right.
\label{NS-M}%
\end{equation}
At each time step $t_i$, the forward process $X_{i}^{\pi,m,k}$ is generated as below:
$$
X_{i+1}^{\pi,m,k} =X_{i}^{\pi,m} +\mu(t_i,X_i^{\pi,m})h+\sigma(t_i,X_i^{\pi,m})\Delta W_i^{m,k},
$$
where
$\Delta W_i^{m,k} \sim \mathcal{N}(0, h)$ is independently and identically distributed as $\Delta W_i^{m} $;
 $k = 1,2,\cdots,K_i$; 
$K_i \in \mathbb{N}$ denotes the number of trajectories in the Monte Carlo simulation.

Now, the approximation algorithm is implemented in detail.
The terms on the right hand side of (\ref{NS-M}) are written as
\begin{equation}
\left\{
\begin{aligned}
Y_{i}^{\pi,m}
= & \frac{1}{K_i}\sum\limits_{k=1}^{K_i} Y_{i+1}^{\pi,m,k}+hf(t_{i},X_{i}^{\pi,m},Y_{i}^{\pi,m},Z_{i}^{\pi,m})+R_{i,1}^{y},\\
Z_{i}^{\pi,m} 
= & \frac{1}{K_i}\sum\limits_{k=1}^{K_i} \left[Y_{i+1}^{\pi,m,k}\frac{\left(\Delta W_{i}^{m,k}\right)^\top}{h}\right]+R_{i,1}^{z},\\
\end{aligned}
\right.
\label{NS-MC}%
\end{equation}
by Monte Carlo method, where
$Y_{i+1}^{\pi,m,k} $ is independently and identically distributed as $Y_{i+1}^{\pi,m}$;
\begin{equation}
\begin{aligned}
R_{i,1}^{y} =& \mathbb{E}_{i}\Big[Y_{i+1}^{\pi,m}\Big]
-\frac{1}{K_i}\sum\limits_{k=1}^{K_i} Y_{i+1}^{\pi,m,k},\\
R_{i,1}^{z} =& \mathbb{E}_{i}\left[Y_{i+1}^{\pi,m}\frac{(\Delta W_{i}^m)^\top}{h}\right]
- \frac{1}{K_i}\sum\limits_{k=1}^{K_i} Y_{i+1}^{\pi,m,k}\frac{\left(\Delta W_{i}^{m,k}\right)^\top}{h}.
\end{aligned}
\nonumber
\end{equation}
Consequently, following the Monte Carlo simulation, the computational formula for  $(Y_{i}^{\pi,m},Z_{i}^{\pi,m})$ is given by
\begin{equation}
\left\{\begin{aligned}
Y_{i,K_i}^{\pi,m} & =\frac{1}{K_i}\sum\limits_{k=1}^{K_i}Y_{i+1, K_i}^{\pi,m, k}+hf\left(t_i, X_{i}^{\pi,m}, Y_{i, K_i}^{\pi,m} , Z_{i, K_i}^{\pi,m}   \right),\\
Z_{i, K_i}^{\pi,m}  & =\frac{1}{K_i}\sum\limits_{k=1}^{K_i}Y_{i+1, K_i}^{\pi,m, k} \frac{(\Delta W_{i}^{m,k})^\top}{h}.
\end{aligned}\right.
\label{Deep-NN-E-k}
\end{equation}
Replacing the two random variables in local Monte Carlo approximates in (\ref{Deep-NN-E-k}) with
two deep feedforward neural networks, namely
\begin{equation}
\left\{
\begin{aligned}
Y_{i,K_i}^{\pi,m}
= & \mathcal{NN}_{y,i,1+d, n}^{\varrho, L, \mathbf{n}} (p_i^m,\theta^0_{y,i})+ R_{i,2}^{y,0},\\
Z_{i,K_i}^{\pi,m}
= & \mathcal{NN}_{z,i,1+d, n\times d}^{\varrho, L, \mathbf{n}}(p^m_i,\theta^0_{z,i})+ R_{i,2}^{z,0},\\
\end{aligned}
\right.
\label{NS-MC-1}%
\end{equation}
where  
$p_i^m = (t_i,X_{i}^{\pi,m})$; $\theta^0_{y,i} $ and
$\theta^0_{z,i} $ denote the initial parameters of the two neural networks at each time step;
\begin{equation}
\begin{aligned}
R_{i,2}^{y,0}=& Y_{i,K_i}^{\pi,m}
-\mathcal{NN}_{y,i,1+d, n}^{\varrho, L, \mathbf{n}}(p_i^m,\theta^0_{y,i}),\\
R_{i,2}^{z,0} =&Z_{i,K_i}^{\pi,m}
-\mathcal{NN}_{z,i,1+d, n\times d}^{\varrho, L, \mathbf{n}}(p_i^m,\theta^0_{z,i}).
\end{aligned}
\nonumber
\end{equation}
To train the deep feedforward neural networks, we introduce the expectation loss functions in the following forms
\begin{equation}
\left\{
\begin{aligned}
F_{y,i}^{r_j}(\theta_{y,i}^{r_j})= & \frac{1}{M}\sum_{m=1}^{M} \left| S_{y,i}\left(p_i^m, X_{i+1}^{\pi, m} \right)
 - \mathcal{NN}_{y,i,1+d, n}^{\varrho, L, \mathbf{n}}(p_i^m,\theta_{y,i}^{r_j})\right|^2,\\
F_{z,i}^{r_j}(\theta_{z,i}^{r_j})= & \frac{1}{M}\sum_{m=1}^{M} \left| S_{z,i}\left(p_i^m,\Delta W_{i}^{m}, X_{i+1}^{\pi, m} \right)
 - \mathcal{NN}_{z,i,1+d, n\times d}^{\varrho, L, \mathbf{n}}(p_i^m,\theta_{z,i}^{r_j})
\right|^2,
\end{aligned}
\right.
\label{NS-MC-3}%
\end{equation}
where
$p_{i}^{ m, k}= (t_i,X_{i}^{\pi,m,k})$;
$r_j\in \mathbb{N}^+$ denotes the number of iterations of the stochastic gradient descent (SGD) method;
$\theta^{r_j}_{y,i}  = \theta_{y,i}^{r_j-1}-\rho_{y,i}^{r_j-1}\nabla_{\theta_{y,i}^{r_j-1}}F_{y,i}^{r_j-1}(\theta_{y,i}^{r_j-1}),$
$\theta^{r_j}_{z,i}  = \theta_{z,i}^{r_j-1}-\rho_{z,i}^{r_j-1}\nabla_{\theta_{z,i}^{r_j-1}}F_{z,i}^{r_j-1}(\theta_{z,i}^{r_j-1})$;
$\rho_{y,i}^{r_j},\rho_{z,i}^{r_j}\in(0,+\infty)$ denote the learning rates;
suppose that the SGD algorithm is capable of converging to the optimal solution:
$\theta^{*}_{y,i} = \arg \min\limits_{\theta_{y,i}^{r_j}} F_{y,i}^{r_j}(\theta_{y,i}^{r_j})$,
$\theta^{*}_{z,i} = \arg \min\limits_{\theta_{z,i}^{r_j}} F_{z,i}^{r_j}(\theta_{z,i}^{r_j})$;
\begin{equation}
\begin{aligned}
S_{z,i}\left(p_i^m,\Delta W_{i}^{m}, X_{i+1}^{\pi, m} \right)
&=\frac{1}{K_i}\sum\limits_{k=1}^{K_i} \mathcal{NN}_{y,i+1,1+d, n}^{\varrho, L, \mathbf{n}}(p_{i+1}^{m,k},\theta_{y,i+1}^*) \frac{\left(\Delta W_{i}^{m,k}\right)^\top}{h},\\
S_{y,i}\left(p_i^m, X_{i+1}^{\pi, m} \right)
&= \frac{1}{K_i}\sum\limits_{k=1}^{K_i} \mathcal{NN}_{y,i+1,1+d, n}^{\varrho, L, \mathbf{n}}(p_{i+1}^{m,k},\theta_{y,i+1}^*)\\
&\quad+hf(t_{i},X_{i}^{\pi,m},\mathcal{NN}_{y,i,1+d, n}^{\varrho, L, \mathbf{n}}(p_{i}^{m},\theta_{y,i}^{r_j}),\mathcal{NN}_{z,i,1+d, n\times d}^{\varrho, L, \mathbf{n}}(p_{i}^{m},\theta_{z,i}^*)).
\end{aligned}
\nonumber
\end{equation}
Thus, by the optimization of the SGD method, (\ref{NS-MC-1}) is rewritten as
\begin{equation}
\left\{
\begin{aligned}
Y_{i,K_i}^{\pi,m}
= &  \mathcal{NN}_{y,i,1+d, n}^{\varrho, L, \mathbf{n}}(p^{m}_{i},\theta^{*}_{y,i})+ R_{i,2}^{y},\\
Z_{i,K_i}^{\pi,m}
= &\mathcal{NN}_{z,i,1+d, n\times d}^{\varrho, L, \mathbf{n}}(p^{m}_{i},\theta^{*}_{z,i})+ R_{i,2}^{z},
\end{aligned}
\right.
\label{NS-MC-2-1}%
\end{equation}
with 
\begin{equation}
\begin{aligned}
R_{i,2}^{y}&=
Y_{i,K_i}^{\pi,m}- \mathcal{NN}_{y,i,1+d, n}^{\varrho, L, \mathbf{n}}(p_{i}^m,\theta_{y,i}^*),\\
R_{i,2}^{z}&=
Z_{i,K_i}^{\pi,m}
 - \mathcal{NN}_{z,i,1+d, n\times d}^{\varrho, L, \mathbf{n}}(p_{i}^m,\theta_{z,i}^*).
\end{aligned}
\nonumber
\end{equation}

Therefore, based on the DBR method, we define the fully discrete approximations
\(\{\mathcal{Y}_i\}_{i=0}^{N}\) and
\(\{\mathcal{Z}_i\}_{i=0}^{N-1}\) for the analytical solutions
\((Y_{t_i},Z_{t_i})\), \(i=0,1,\ldots,N\), \(m=1,2,\ldots,M\), as follows:
\begin{enumerate}
\item
the terminal condition is
\[
\mathcal{Y}_N=g(X_N^{\pi,m}).
\]
No terminal value \(\mathcal{Z}_N\) is required by the backward scheme.

\item for \(0\le i<N\), the transition from \(i+1\) to \(i\) is given by
\begin{equation}
\left\{
\begin{aligned}
\mathcal{Y}_{i}
= &\ \mathcal{NN}_{y,i,1+d,n}^{\varrho,L,\mathbf{n}}
(p^{m}_{i},\theta^{*}_{y,i}),\\
\mathcal{Z}_{i}
= &\ \mathcal{NN}_{z,i,1+d,n\times d}^{\varrho,L,\mathbf{n}}
(p^{m}_{i},\theta^{*}_{z,i}).
\end{aligned}
\right.
\label{NS-DMC}%
\end{equation}
\end{enumerate}

For notational consistency in the subsequent error analysis, we use the following
terminal-network convention:
\begin{equation}
\begin{aligned}
\mathcal{NN}_{y,N,1+d,n}^{\varrho,L,\mathbf{n}}
(p_N^m,\theta_{y,N}^{*})
:=g(X_N^{\pi,m}).
\end{aligned}
\label{NS-DMC-terminal-convention}
\end{equation}
Here \(\mathcal{NN}_{y,N,1+d,n}^{\varrho,L,\mathbf{n}}\) is only a notational
convention at the terminal time and no neural network is trained at \(t_N\).

In summary, the DBR algorithm is provided to compute the numerical solutions $\mathcal{Y}_{i}$ and $\mathcal{Z}_{i}$
(see \textbf{Algorithm 1}).

\begin{algorithm}[H]
\caption{DBR method for FBSDEs}
\label{Alg-DBR}
Grid \(0=t_0<t_1<\cdots<t_N=T\) with step \(h=T/N\);
the number \(M\) of base Monte Carlo paths; the numbers \(K_i\) of inner
Monte Carlo samples.
Trained networks
\(\{\mathcal{NN}_{y,i},\mathcal{NN}_{z,i}\}_{i=0}^{N-1}\) and the final estimate
\(\mathcal{Y}_0\).

Initialize the network parameters
\(\theta_{y,i},\theta_{z,i}\) for \(i=0,\ldots,N-1\), e.g., by Xavier initialization;

\textbf{Forward simulation of base paths:}

\For{\(m=1,\ldots,M\)}{
Set \(X_0^{\pi,m}=x_0\);

\For{\(i=0,\ldots,N-1\)}{
Generate \(\Delta W_i^m\sim N(0,hI_d)\);

Compute
\[
X_{i+1}^{\pi,m}
\leftarrow
X_i^{\pi,m}
+
\mu(t_i,X_i^{\pi,m})h
+
\sigma(t_i,X_i^{\pi,m})\Delta W_i^m .
\]
}
}

\textbf{Backward training:}

\For{\(i=N-1,\ldots,0\)}{

\For{\(m=1,\ldots,M\)}{

\For{\(k=1,\ldots,K_i\)}{
Generate \(\Delta W_i^{m,k}\sim N(0,hI_d)\);

Compute the inner next state
\[
X_{i+1}^{\pi,m,k}
\leftarrow
X_i^{\pi,m}
+
\mu(t_i,X_i^{\pi,m})h
+
\sigma(t_i,X_i^{\pi,m})\Delta W_i^{m,k}.
\]

Set
\[
p_i^m:=(t_i,X_i^{\pi,m}),
\qquad
p_{i+1}^{m,k}:=(t_{i+1},X_{i+1}^{\pi,m,k}).
\]

\eIf{\(i=N-1\)}{
\[
\mathcal{Y}_{i+1}^{m,k}
\leftarrow
g(X_{i+1}^{\pi,m,k});
\]
}{
\[
\mathcal{Y}_{i+1}^{m,k}
\leftarrow
\mathcal{NN}_{y,i+1,1+d,n}^{\varrho,L,\mathbf n}
(p_{i+1}^{m,k},\theta_{y,i+1}^{*});
\]
}
}
}

Define the \(Z\)-loss
\[
\mathcal{L}_{z,i}(\theta_{z,i})
:=
\frac1M\sum_{m=1}^{M}
\left|
\frac1{K_i}\sum_{k=1}^{K_i}
\mathcal{Y}_{i+1}^{m,k}
\frac{(\Delta W_i^{m,k})^\top}{h}
-
\mathcal{NN}_{z,i,1+d,n\times d}^{\varrho,L,\mathbf n}
(p_i^m,\theta_{z,i})
\right|^2 .
\]

Update \(\theta_{z,i}\) by SGD, and denote the obtained optimal parameter by
\(\theta_{z,i}^{*}\);

Define the \(Y\)-loss
\[
\mathcal{L}_{y,i}(\theta_{y,i})
:=
\frac1M\sum_{m=1}^{M}
\Bigg|
\frac1{K_i}\sum_{k=1}^{K_i}
\mathcal{Y}_{i+1}^{m,k}
+
h f\left(
t_i,
X_i^{\pi,m},
\mathcal{NN}_{y,i,1+d,n}^{\varrho,L,\mathbf n}
(p_i^m,\theta_{y,i}),
\mathcal{NN}_{z,i,1+d,n\times d}^{\varrho,L,\mathbf n}
(p_i^m,\theta_{z,i}^{*})
\right)\] \\
$ -\mathcal{NN}_{y,i,1+d,n}^{\varrho,L,\mathbf n}
(p_i^m,\theta_{y,i})
\Bigg|^2.$

Update \(\theta_{y,i}\) by SGD, and denote the obtained optimal parameter by
\(\theta_{y,i}^{*}\);

Set, for \(m=1,\ldots,M\),
\[
\mathcal{Y}_i^m
\leftarrow
\mathcal{NN}_{y,i,1+d,n}^{\varrho,L,\mathbf n}
(p_i^m,\theta_{y,i}^{*}),
\qquad
\mathcal{Z}_i^m
\leftarrow
\mathcal{NN}_{z,i,1+d,n\times d}^{\varrho,L,\mathbf n}
(p_i^m,\theta_{z,i}^{*}).
\]
}

\BlankLine
Final estimation:
\[
\mathcal{Y}_0
=
\mathcal{NN}_{y,0,1+d,n}^{\varrho,L,\mathbf n}
((t_0,x_0),\theta_{y,0}^{*}).
\]
\end{algorithm}

The DBDP method in \cite{HCPHWX20} uses $\mathcal{Y}_{i+1}^*
+hf(t_i,X_i^\pi,\mathcal{Y}_i,\mathcal{Z}_i ) -\mathcal{Z}_i\Delta W_i$ as the label for $(\mathcal{Y}_i(\theta_i), \mathcal{Z}_i(\theta_i))$ ( or $(\mathcal{Y}_i(\theta_i),  \sigma^\top(t_i,\cdot) D_x\mathcal{Y}_i(\theta_i)$)
at each time step $t_i$. $\mathcal{Y}_i$ and $\mathcal{Z}_i$
are $\mathcal{F}_{t_i}$ measurable while the information of Brownian motion at moment $t_{i+1}$ is contained in the label.
This causes some fluctuations in the estimates of $\mathcal{Y}_i$ and $\mathcal{Z}_i$.
Our DBR method employs $\mathbb{E}_i\left[\mathcal{Y}_{i+1}^*\right]
+hf(t_i,X_i^\pi,\mathcal{Y}_i,\mathcal{Z}_i^* )$ and $\mathbb{E}_i\left[\mathcal{Y}_{i+1}^*\frac{\Delta W_i^\top}{h} \right]$ as the labels for $\mathcal{Y}_i$ and $\mathcal{Z}_i$ respectively. This can smooth the noise  $\Delta W_i$ by
averaging, and then makes the estimation  $\mathcal{Y}_i$ and $\mathcal{Z}_i$ more stable at each time step.
The DBDP method relies on the information of single-path  for labeling, which is random and prone to  learning the incorrect solutions of PDEs.
Our DBR method, by
multi-path averaging, offsets the noise of single-path. This makes our method  more robust and having stronger generalization ability (especially in high-dimensional or complex PDEs).

\subsection{Error analysis}

We now estimate the error of the numerical solutions $(\mathcal{Y}_i,\mathcal{Z}_i)_{0\le i\le N-1}$
and the analytical solutions $(Y_{t_i},Z_{t_i})_{0\le i\le N-1}$.
This result provides an error estimate of the DBR method (\ref{NS-DMC})
in the $L^2$ sense. 

The standard assumptions which guarantee the existence and uniqueness
on the coefficients of the FBSDEs (\ref{dFBSDE}) associated to the semilinear PDEs (\ref{PDE})
are made as following:
\begin{description}
\item[(i)]
$x_0$ is square integrable.
\item[(ii)]
$\mu$ and $\sigma$ are Lipschitz continuous of the spatial variable $x$;
assume
\[
\sup\limits_{0\le t\le T}\{|\mu(t,0)|+|\sigma(t,0)|\}\leq L,
\]
where  $L>0$  denotes the Lipschitz constant.
\item[(iii)]
 $f$ is $\frac{1}{2}-$H\"{o}lder continuous with respect to the time variable $t$ and
uniformly Lipschitz continuous in all other variables
$$
|f(t_1,x_1,y_1,z_{1})-f(t_2,x_2,y_2,z_{2})|\le
L(|t_{1}-t_{2}|^{\frac{1}{2}}+|x_{1}-x_{2}|+|y_{1}-y_{2}|+|z_{1}-z_{2}|),
$$
for all $t\in[0,T],$
$x_{1},x_{2}\in \mathbb{R}^d$, $y_{1},y_{2} \in\mathbb{R}^n$ and $z_{1},z_{2}\in\mathbb{R}^{n\times d}$;
and suppose
\[
\sup\limits_{0\le t\le T}\{|f(t,0,0,0)|\}\leq L.
\]
\item[(iv)]
$g$ is a linear growth function.
\end{description}
\medskip
\noindent
In the error analysis below, the same time-step condition
\eqref{time-step-condition} is imposed. Consequently, all implicit equations
appearing in the DBR scheme and in the auxiliary processes are well-defined. 

Now, we investigate the errors of the scheme (\ref{NS-DMC}) and define, for $i=N-1, N-2,\cdots,0$
\begin{equation}
\left\{\begin{aligned}
\mathcal{U}_{i}^m & :=\mathbb{E}_i\left[\mathcal{NN}_{y,i+1,1+d, n}^{\varrho, L, \mathbf{n}}(p_{i+1}^{m},\theta^{*}_{y,i+1})+hf\left(t_i, X_{i}^{\pi,m}, \mathcal{U}_{i}^m, \mathcal{V}_{i}^m\right)  \right],\\
\mathcal{V}_{i}^m  & :=\mathbb{E}_i\left[\mathcal{NN}_{y,i+1,1+d, n}^{\varrho, L, \mathbf{n}}(p_{i+1}^{m} ,\theta^{*}_{y,i+1})\frac{(\Delta W_{i}^{m})^\top}{h}\right],
\end{aligned}\right.
\label{Deep-NN-E}
\end{equation}
and by the Markov property of the discretized forward process
$\{X_i^{\pi,m}\}_{0\le i\le N}$, we have
\begin{equation}
\begin{aligned}
\mathcal{U}_{i}^m = \widetilde{u}(t_i,X^{\pi,m}_i),
\quad
\mathcal{V}_{i}^m = \widetilde{v}(t_i,X^{\pi,m}_i),
\end{aligned}
\label{Deep-NN-E--YZK-1}
\end{equation}
where $\widetilde{u}$ and $\widetilde{v}$ are some deterministic (but unknown) functions;
\begin{equation}
\left\{\begin{aligned}
\widetilde{\mathcal{U}}_{i}^m & :=\frac{1}{K_i}\sum\limits_{k=1}^{K_i}\mathcal{NN}_{y,i+1,1+d, n}^{\varrho, L, \mathbf{n}}(p_{i+1}^{m,k},\theta^{*}_{y,i+1})+hf\left(t_i, X_{i}^{\pi,m}, \widetilde{\mathcal{U}}_{i}^m, \widetilde{\mathcal{V}}_{i}^m\right) ,\\
\widetilde{\mathcal{V}}_{i}^m  & :=\frac{1}{K_i}\sum\limits_{k=1}^{K_i}\mathcal{NN}_{y,i+1,1+d, n}^{\varrho, L, \mathbf{n}}(p_{i+1}^{m,k} ,\theta^{*}_{y,i+1})\frac{(\Delta W_{i}^{m,k})^\top}{h}.
\end{aligned}\right.
\label{Deep-NN-E--YZK}
\end{equation}
Let us introduce
\begin{equation}
\begin{aligned}
&\varepsilon_{i}^{y} = \inf\limits_{\mathcal{NN}_{y,i,1+d, n}^{\varrho, L, \mathbf{n}}}\mathbb{E}\left[\left|\widetilde{u}(t_i,X^{\pi,m}_i)-\mathcal{NN}_{y,i,1+d, n}^{\varrho, L, \mathbf{n}}(p_{i}^{m},\theta^{r_j}_{y,i})\right|^2\right],\\
&\varepsilon_{i}^{z} = \inf\limits_{\mathcal{NN}_{z,i,1+d, n\times d}^{\varrho, L, \mathbf{n}}}\mathbb{E}\left[\left|\widetilde{v}(t_i,X^{\pi,m}_i)-\mathcal{NN}_{z,i,1+d, n\times d}^{\varrho, L, \mathbf{n}}(p_{i}^{m},\theta^{r_j}_{z,i})\right|^2\right],
\end{aligned}
\label{Deep-NN-E--YZK-21}
\end{equation}
and
\begin{equation}
\begin{aligned}
&\widetilde{\varepsilon}_{i}^{y} = \inf\limits_{\mathcal{NN}_{y,i,1+d, n}^{\varrho, L, \mathbf{n}}}\mathbb{E}\left[\left|\widetilde{\mathcal{U}}_i^{m}-\mathcal{NN}_{y,i,1+d, n}^{\varrho, L, \mathbf{n}}(p_{i}^{m},\theta^{r_j}_{y,i})\right|^2\right],\\
&\widetilde{\varepsilon}_{i}^{z} = \inf\limits_{\mathcal{NN}_{z,i,1+d, n\times d}^{\varrho, L, \mathbf{n}}}\mathbb{E}\left[\left|\widetilde{\mathcal{V}}_i^{m}-\mathcal{NN}_{z,i,1+d, n\times d}^{\varrho, L, \mathbf{n}}(p_{i}^{m},\theta^{r_j}_{z,i})\right|^2\right],
\end{aligned}
\label{Deep-NN-E--YZK-2}
\end{equation}

By a slight abuse of notation, we henceforth write $\widetilde{\mathcal{V}}_{i}^m$
to denote $S_{z,i}\left(p_i^m,\Delta W_{i}^{m}, X_{i+1}^{\pi, m} \right)$,
to shift the focus from the specific functional form of the network to its role as an approximated function.

\noindent\textbf{Population-loss convention.}
The loss functions in \eqref{NS-MC-3} are empirical losses used for the
implementation of Algorithm 1. In the following error analysis, however, we work
with their population counterparts. More precisely, the expectation
\(\mathbb{E}[\cdot]\) is taken with respect to all sources of randomness involved in
\(X_i^{\pi,m}\), \(\widetilde{\mathcal{U}}_i^m\), and
\(\widetilde{\mathcal{V}}_i^m\). For fixed \(i\) and \(r_j\), we define
\begin{equation}
\begin{aligned}
\mathcal{F}_{z,i}^{r_j}(\theta_{z,i}^{r_j})
:=&
\mathbb{E}\left[
\left|
\widetilde{\mathcal{V}}_{i}^m
-
\mathcal{NN}_{z,i,1+d,n\times d}^{\varrho,L,\mathbf{n}}
(p_i^m,\theta_{z,i}^{r_j})
\right|^2
\right],
\end{aligned}
\label{POP-LOSS-Z}
\end{equation}
and, after fixing the population minimizer \(\theta_{z,i}^{*}\),
\begin{equation}
\begin{aligned}
\mathcal{F}_{y,i}^{r_j}(\theta_{y,i}^{r_j})
:=&
\mathbb{E}\Bigg[
\bigg|
\widetilde{\mathcal{U}}_{i}^m
+h\Big(
f\big(t_i,X_i^{\pi,m},
\mathcal{NN}_{y,i,1+d,n}^{\varrho,L,\mathbf{n}}
(p_i^m,\theta_{y,i}^{r_j}),
\mathcal{NN}_{z,i,1+d,n\times d}^{\varrho,L,\mathbf{n}}
(p_i^m,\theta_{z,i}^{*})\big)
\\
&\qquad\qquad
-
f\big(t_i,X_i^{\pi,m},
\widetilde{\mathcal{U}}_{i}^m,
\widetilde{\mathcal{V}}_{i}^m\big)
\Big)
-
\mathcal{NN}_{y,i,1+d,n}^{\varrho,L,\mathbf{n}}
(p_i^m,\theta_{y,i}^{r_j})
\bigg|^2
\Bigg].
\end{aligned}
\label{POP-LOSS-Y}
\end{equation}
Throughout the proof of Theorem 3.1, the symbols
\(F_{y,i}^{r_j}\) and \(F_{z,i}^{r_j}\) are understood as the population losses
\(\mathcal{F}_{y,i}^{r_j}\) and \(\mathcal{F}_{z,i}^{r_j}\), respectively. The empirical
losses in \eqref{NS-MC-3} are Monte Carlo approximations of these population
losses. Therefore, Theorem 3.1 should be interpreted as an approximation-error
estimate under exact minimization of the population losses. The additional errors
arising from replacing the population losses by their empirical counterparts and from
non-exact stochastic optimization are not included in this theorem; they can be
incorporated separately as generalization and optimization residuals.
\medskip

\begin{theorem} \label{ae-DMCE}
Suppose the assumptions (i)-(iv) hold.
Let $(Y_{t_i},Z_{t_i})$
and $(\mathcal{Y}_{i},\mathcal{Z}_{i})$ be solutions of the FBSDEs  (\ref{dFBSDE}) and solutions of the DBR method
(\ref{NS-DMC}) respectively.
Then, for \(h\) small enough and satisfying \eqref{time-step-condition}, we have
\begin{equation}
\begin{aligned}
&\max\limits_{0\le i\le N-1}\mathbb{E}\left[\left|Y_{t_i}-\mathcal{Y}_{i}\right|^2\right]
+\mathbb{E}\left[\sum\limits_{i=0}^{N-1}\int_{t_{i}}^{t_{i+1}} |Z_{s}-\mathcal{Z}_{i}|^2ds \right]\\
&\le C\mathbb{E}\left[\left| g(X_T)-g(X_N^{\pi,m})\right|^2\right] + Ch
+C\mathbb{E}\left[\sum\limits_{i=0}^{N-1}\int_{t_{i}}^{t_{i+1}} |Z_{s}-\bar{Z}_{t_i}|^2ds \right]\\
&\quad+C\sum\limits_{i=0}^{N-1}\mathbb{E}\left[
N\varepsilon_{i}^{y}
+\varepsilon_{i}^{z}
+\frac{N}{K}\right],
\nonumber
\end{aligned}
\end{equation}
where $K=\min\limits_{0\le i\le N-1}\{K_i\}$; 
 $\bar{Z}_{t_i}=\mathbb{E}_i\left[\frac{1}{h} \int_{t_i}^{t_{i+1}} Z_t \mathrm{~d} t\right]$; 
$C$ represents a positive generic constant independent of $\pi$, which may depend on $d$ and change from line to line.
\end{theorem}
\begin{proof}
Step 1.
On the uniform discrete mesh $\pi$, we can write the BSDE  in (\ref{dFBSDE})
at the mesh points $t_i$ as follows
\begin{equation}
Y_{t_i}= Y_{t_{i+1}} + \int_{t_{i}}^{t_{i+1}} f_sds - \int_{t_{i}%
}^{t_{i+1}} Z_{s} d W_{s}, \label{FBSDE-t-i}%
\end{equation}
where $f_s = f(s,X_{s},Y_{s},Z_{s})$. From (\ref{Deep-NN-E}) and (\ref{FBSDE-t-i}), we have, for $i\in\{0,1,\ldots,N-1\}$
\begin{equation}
\begin{aligned}
Y_{t_i}-\mathcal{U}^m_{i}=&\mathbb{E}_{i}\left[ Y_{t_{i+1}}-\mathcal{NN}_{y,i+1,1+d, n}^{\varrho, L, \mathbf{n}}(p_{i+1}^{m},\theta^{*}_{y,i+1}) \right]\\
&+ \mathbb{E}_{i}\left[\int_{t_{i}}^{t_{i+1}} \left(f_s -f\left(t_i, X_{i}^{\pi,m},\mathcal{U}^m_{i}, \mathcal{V}_{i}^m \right)\right)ds \right].
\label{Deep-NN-E-1}
\end{aligned}
\end{equation}
Following a similar proof strategy for the convergence of the Euler method in \cite{BBTN04,ZJ04}, we have
\begin{equation}
\begin{aligned}
\mathbb{E}\left[\left|Y_{t_i}-\mathcal{U}^m_{i}\right|^2\right]
\le& (1+C h)\mathbb{E}\left[\left| Y_{t_{i+1}}-\mathcal{NN}_{y,i+1,1+d, n}^{\varrho, L, \mathbf{n}}(p_{i+1}^m,\theta^{*}_{y,i+1})\right|^2\right] + Ch^2\\
&+C\left(\mathbb{E}\left[\int_{t_{i}}^{t_{i+1}} |Z_{s}-\bar{Z}_{t_i}|^2ds \right] +h\mathbb{E}\left[\int_{t_{i}}^{t_{i+1}}f_s^2ds\right]\right).
\label{Deep-NN-E-5}
\end{aligned}
\end{equation}
By Young inequality in the form: $(a+b)^2\ge (1-h)a^2 + (1-\frac{1}{h})b^2\ge (1-h)a^2 -\frac{1}{ h}b^2,$ for $a,b \in \mathbb{R}$, we obtain
\begin{equation}
\begin{aligned}
\mathbb{E}\left[\left|Y_{t_i}-\mathcal{U}^m_{i}\right|^2\right]
\ge& (1- h)\mathbb{E}\left[\left| Y_{t_{i}}-\mathcal{Y}_{i}\right|^2\right]
-\frac{1}{ h}\mathbb{E}\left[\left|\mathcal{Y}_{i}-\mathcal{U}^m_{i}\right|^2\right].
\label{Deep-NN-E-6}
\end{aligned}
\end{equation}
Inserting  (\ref{Deep-NN-E-6}) into (\ref{Deep-NN-E-5}), we have
\begin{equation}
\begin{aligned}
\mathbb{E}\left[\left|Y_{t_i}-\mathcal{Y}_{i}\right|^2\right]
&\le (1+C h)\mathbb{E}\left[\left| Y_{t_{i+1}}-\mathcal{Y}_{i+1}\right|^2\right]
+ Ch^2\\
&\quad+C\left(\mathbb{E}\left[\int_{t_{i}}^{t_{i+1}} |Z_{s}-\bar{Z}_{t_i}|^2ds \right]
+h\mathbb{E}\left[\int_{t_{i}}^{t_{i+1}}f_s^2ds\right]
+\frac{1}{ h}\mathbb{E}\left[\left|\mathcal{Y}_{i}-\mathcal{U}^m_{i}\right|^2\right]\right).
\label{Deep-NN-E-7}
\end{aligned}
\end{equation}
With the help of the discrete Gronwall's lemma, we derive
\begin{equation}
\begin{aligned}
\max\limits_{0\le i\le N-1}\mathbb{E}\left[\left|Y_{t_i}-\mathcal{Y}_{i}\right|^2\right]
&\le C\mathbb{E}\left[\left| g(X_T)-g(X_N^{\pi,m})\right|^2\right]
+ Ch\\
&\quad+C\left(\mathbb{E}\left[\sum\limits_{i=0}^{N-1}\int_{t_{i}}^{t_{i+1}} |Z_{s}-\bar{Z}_{t_i}|^2ds \right]
+\frac{1}{ h}\sum\limits_{i=0}^{N-1}\mathbb{E}\left[\left|\mathcal{Y}_{i}-\mathcal{U}^m_{i}\right|^2\right]\right).
\label{Deep-NN-E-8}
\end{aligned}
\end{equation}

Step 2.  By Young inequality, we get
\begin{equation}
\begin{aligned}
\mathbb{E}\left[\left|\mathcal{Y}_{i}-\mathcal{U}^m_{i}\right|^2\right]
\le&  2\mathbb{E}\left[\left|\mathcal{Y}_{i}-\widetilde{\mathcal{U}}_{i}^m\right|^2\right]
+2\mathbb{E}\left[\left|\widetilde{\mathcal{U}}_{i}^m
-\mathcal{U}^m_{i} \right|^2\right].
\label{Deep-NN-E-9}
\end{aligned}
\end{equation}

In what follows, we handle each term in (\ref{Deep-NN-E-9}) separately.

$\blacktriangleright$ Term $\mathbb{E}\left[\left|\mathcal{Y}_{i}-\widetilde{\mathcal{U}}_{i}^m\right|^2\right]$. 
From (\ref{NS-MC-3}) and (\ref{Deep-NN-E--YZK}), Young inequality in the form
\[
|a+b|^2\le (1+\gamma h)|a|^2+\left(1+\frac{1}{\gamma h}\right)|b|^2,
\]
for \(a,b\) in a Euclidean space and \(\gamma>0\), together with the Lipschitz condition with respect to \(f\), we deduce
\begin{equation}
\begin{aligned}
F_{z,i}^{r_j}(\theta_{z,i}^{r_j})
= & \mathbb{E} \left[\left| \widetilde{\mathcal{V}}_{i}^m
 - \mathcal{NN}_{z,i,1+d, n\times d}^{\varrho, L, \mathbf{n}}(p_i^m,\theta_{z,i}^{r_j})
\right|^2\right].
\end{aligned}
\label{NS-MC-3-Z}%
\end{equation}
Moreover, by the optimality of \(\theta_{z,i}^{*}\), we have, for any \(\theta_{z,i}^{r_j}\),
\begin{equation}
\begin{aligned}
F_{z,i}^{*}(\theta_{z,i}^{*})
:= & \mathbb{E} \left[\left| \widetilde{\mathcal{V}}_{i}^m
 - \mathcal{NN}_{z,i,1+d, n\times d}^{\varrho, L, \mathbf{n}}(p_i^m,\theta_{z,i}^{*})
\right|^2\right]  \\
\leq &\ F_{z,i}^{r_j}(\theta_{z,i}^{r_j}).
\end{aligned}
\label{NS-MC-3-Zstar}%
\end{equation}
Furthermore,
\begin{equation}
\begin{aligned}
F_{y,i}^{r_j}(\theta_{y,i}^{r_j})
= & \mathbb{E} \Bigg[\bigg| \widetilde{\mathcal{U}}_{i}^m
+h\Big( f(t_{i},X_{i}^{\pi,m},
\mathcal{NN}_{y,i,1+d, n}^{\varrho, L, \mathbf{n}}(p_{i}^{m},\theta_{y,i}^{r_j}),
\mathcal{NN}_{z,i,1+d, n\times d}^{\varrho, L, \mathbf{n}}(p_{i}^{m},\theta_{z,i}^*))\\
&\quad - f(t_i, X_{i}^{\pi,m},\widetilde{\mathcal{U}}_{i}^m,
\widetilde{\mathcal{V}}_{i}^m)\Big)
- \mathcal{NN}_{y,i,1+d, n}^{\varrho, L, \mathbf{n}}(p_i^m,\theta_{y,i}^{r_j})
\bigg|^2\Bigg]\\
 \le & (1+Ch)\mathbb{E} \left[\left| \widetilde{\mathcal{U}}_{i}^m
- \mathcal{NN}_{y,i,1+d, n}^{\varrho, L, \mathbf{n}}(p_i^m,\theta_{y,i}^{r_j})\right|^2\right]\\
&+Ch \mathbb{E}\left[ \left| \widetilde{\mathcal{V}}_{i}^m
- \mathcal{NN}_{z,i,1+d, n\times d}^{\varrho, L, \mathbf{n}}(p_{i}^{m},\theta_{z,i}^*)\right|^2\right].
\end{aligned}
\label{NS-MC-3--Y1}%
\end{equation}
Adding \(hF_{z,i}^{*}(\theta_{z,i}^{*})\) to both sides of (\ref{NS-MC-3--Y1}) and using (\ref{NS-MC-3-Zstar}), we have
\begin{equation}
\begin{aligned}
F_{y,i}^{r_j}(\theta_{y,i}^{r_j})+hF_{z,i}^{*}(\theta_{z,i}^{*})
\le  & (1+Ch)\mathbb{E} \left[\left| \widetilde{\mathcal{U}}_{i}^m
- \mathcal{NN}_{y,i,1+d, n}^{\varrho, L, \mathbf{n}}(p_i^m,\theta_{y,i}^{r_j})\right|^2\right]\\
&+Ch\mathbb{E} \left[\left| \widetilde{\mathcal{V}}_{i}^m
 - \mathcal{NN}_{z,i,1+d, n\times d}^{\varrho, L, \mathbf{n}}(p_i^m,\theta_{z,i}^{r_j})
\right|^2\right].
\end{aligned}
\label{NS-MC-3-YZ}%
\end{equation}
On the other hand, by Young inequality in the form
\[
|a+b|^2\ge (1-\gamma h)|a|^2-\frac{1}{\gamma h}|b|^2,
\]
for \(a,b\) in a Euclidean space, \(\gamma>0\), and \(0<\gamma h<1\), we obtain
\begin{equation}
\begin{aligned}
F_{y,i}^{r_j}(\theta_{y,i}^{r_j})
 \ge & \left(1-\gamma h-\frac{2L^2h}{\gamma}\right)
 \mathbb{E} \left[\left| \widetilde{\mathcal{U}}_{i}^m
- \mathcal{NN}_{y,i,1+d, n}^{\varrho, L, \mathbf{n}}(p_i^m,\theta_{y,i}^{r_j})\right|^2\right]\\
&- \frac{2L^2h}{\gamma}
\mathbb{E} \left[
\left| \widetilde{\mathcal{V}}_{i}^m
- \mathcal{NN}_{z,i,1+d, n\times d}^{\varrho, L, \mathbf{n}}(p_{i}^{m},\theta_{z,i}^*)\right|^2\right].
\end{aligned}
\label{NS-MC-3--Y2}%
\end{equation}
Adding \(hF_{z,i}^{*}(\theta_{z,i}^{*})\) to both sides of (\ref{NS-MC-3--Y2}) gives
\begin{equation}
\begin{aligned}
&F_{y,i}^{r_j}(\theta_{y,i}^{r_j})+hF_{z,i}^{*}(\theta_{z,i}^{*})\\
\ge & \left(1-\gamma h-\frac{2L^2h}{\gamma}\right)
 \mathbb{E} \left[\left| \widetilde{\mathcal{U}}_{i}^m
- \mathcal{NN}_{y,i,1+d, n}^{\varrho, L, \mathbf{n}}(p_i^m,\theta_{y,i}^{r_j})\right|^2\right]\\
&+ h\left(1-\frac{2L^2}{\gamma}\right)
\mathbb{E} \left[
\left| \widetilde{\mathcal{V}}_{i}^m
- \mathcal{NN}_{z,i,1+d, n\times d}^{\varrho, L, \mathbf{n}}(p_{i}^{m},\theta_{z,i}^{*})\right|^2\right].
\end{aligned}
\label{NS-MC-3--Y2-add}%
\end{equation}
Choosing \(\gamma = 4L^2\), we obtain, for \(h\) small enough,
\begin{equation}
\begin{aligned}
F_{y,i}^{r_j}(\theta_{y,i}^{r_j})+hF_{z,i}^{*}(\theta_{z,i}^{*})
 \ge & (1-C h)\mathbb{E} \left[\left| \widetilde{\mathcal{U}}_{i}^m
- \mathcal{NN}_{y,i,1+d, n}^{\varrho, L, \mathbf{n}}(p_i^m,\theta_{y,i}^{r_j})\right|^2\right]\\
&+ \frac{h}{2}
\mathbb{E} \left[
\left| \widetilde{\mathcal{V}}_{i}^m
- \mathcal{NN}_{z,i,1+d, n\times d}^{\varrho, L, \mathbf{n}}(p_{i}^{m},\theta_{z,i}^{*})\right|^2\right].
\end{aligned}
\label{NS-MC-3--Y3}%
\end{equation}
Thus, by the optimality of \(\theta_{y,i}^{*}\), together with (\ref{NS-MC-3-YZ}), we have
\begin{equation}
\begin{aligned}
 &(1-C h)\mathbb{E} \left[\left| \widetilde{\mathcal{U}}_{i}^m
- \mathcal{NN}_{y,i,1+d, n}^{\varrho, L, \mathbf{n}}(p_i^m,\theta_{y,i}^{*})\right|^2\right]\\
&+ \frac{h}{2}
\mathbb{E}\left[
\left| \widetilde{\mathcal{V}}_{i}^m
- \mathcal{NN}_{z,i,1+d, n\times d}^{\varrho, L, \mathbf{n}}(p_{i}^{m},\theta_{z,i}^{*})\right|^2\right]\\
&\le F_{y,i}^{*}(\theta_{y,i}^{*})+hF_{z,i}^{*}(\theta_{z,i}^{*})\\
&\le  (1+Ch)\mathbb{E} \left[\left| \widetilde{\mathcal{U}}_{i}^m
- \mathcal{NN}_{y,i,1+d, n}^{\varrho, L, \mathbf{n}}(p_i^m,\theta_{y,i}^{r_j})\right|^2\right]\\
&\quad+Ch\mathbb{E}\left[ \left| \widetilde{\mathcal{V}}_{i}^m
 - \mathcal{NN}_{z,i,1+d, n\times d}^{\varrho, L, \mathbf{n}}(p_i^m,\theta_{z,i}^{r_j})
\right|^2\right].
\end{aligned}
\label{NS-MC-3--Y4}%
\end{equation}
Taking the infimum over the neural network classes and using (\ref{Deep-NN-E--YZK-2}), we have, for \(h\) small enough,
\begin{equation}
\begin{aligned}
 &\mathbb{E} \left[\left| \widetilde{\mathcal{U}}_{i}^m
- \mathcal{NN}_{y,i,1+d, n}^{\varrho, L, \mathbf{n}}(p_i^m,\theta_{y,i}^{*})\right|^2\right]
+ h
\mathbb{E}\left[
\left| \widetilde{\mathcal{V}}_{i}^m
- \mathcal{NN}_{z,i,1+d, n\times d}^{\varrho, L, \mathbf{n}}(p_{i}^{m},\theta_{z,i}^{*})\right|^2\right]\\
&\le  C\left(\widetilde{\varepsilon}_{i}^{y} +
h\widetilde{\varepsilon}_{i}^{z} \right).
\end{aligned}
\label{NS-MC-3--Y7}%
\end{equation}

$\blacktriangleright$ Term $\mathbb{E}\left[\left|\widetilde{\mathcal{U}}_{i}^m
-\mathcal{U}^m_{i} \right|^2\right]$. From  (\ref{Deep-NN-E}) and (\ref{Deep-NN-E--YZK}), we obtain
\begin{equation}
\begin{aligned}
\mathbb{E}\left[\left|\widetilde{\mathcal{V}}_{i}^m -\mathcal{V}_{i}^m \right|^2\right]
\le \frac{C}{K_i h}\le \frac{C}{Kh},
\end{aligned}
\label{Deep-NN-E--YZK--h1-z}
\end{equation}
where \(K=\min_{0\leq j\leq N-1}K_j\). The above inequality follows from the
Monte Carlo approximation error together with the fact that the random variable
\[
\mathcal{NN}_{y,i+1,1+d,n}^{\varrho,L,\mathbf{n}}(p_{i+1}^{m,k},\theta_{y,i+1}^{*})
\frac{(\Delta W_i^{m,k})^\top}{h}
\]
has a second moment of order \(1/h\).

From  (\ref{Deep-NN-E}), (\ref{Deep-NN-E--YZK}) and (\ref{Deep-NN-E--YZK--h1-z}),
Young inequality in the form
\[
|a+b|^2\leq (1+\gamma h)|a|^2+
\left(1+\frac{1}{\gamma h}\right)|b|^2,
\]
for \(a,b\) in a Euclidean space and \(\gamma>0\), together with the Lipschitz
condition with respect to \(f\), we obtain
\begin{equation}
\begin{aligned}
\mathbb{E}\left[\left|\widetilde{\mathcal{U}}_{i}^m -\mathcal{U}_{i}^m \right|^2\right]
& \le (1+Ch)\mathbb{E}\Bigg[\bigg|
\frac{1}{K_i}\sum\limits_{k=1}^{K_i}
\mathcal{NN}_{y,i+1,1+d, n}^{\varrho, L, \mathbf{n}}
(p_{i+1}^{m,k},\theta^{*}_{y,i+1})\\
&\quad- \mathbb{E}_i\Big[
\mathcal{NN}_{y,i+1,1+d, n}^{\varrho, L, \mathbf{n}}
(p_{i+1}^{m},\theta^{*}_{y,i+1}) \Big]\bigg|^2\Bigg]\\
&\quad+Ch\mathbb{E}\left[
\left|\widetilde{\mathcal{V}}_{i}^m-  \mathcal{V}_{i}^m\right|^2  \right]\\
&\le \frac{C}{K_i}\le \frac{C}{K}.
\end{aligned}
\label{Deep-NN-E--YZK--h1}
\end{equation}
Plugging  (\ref{NS-MC-3--Y7}) and  (\ref{Deep-NN-E--YZK--h1}) into
(\ref{Deep-NN-E-9}), we derive
\begin{equation}
\begin{aligned}
\mathbb{E}\left[\left|\mathcal{Y}_{i}-\mathcal{U}_{i}^m\right|^2\right]
\le&    C\left( \widetilde{\varepsilon}_{i}^{y} +
h\widetilde{\varepsilon}_{i}^{z} +\frac{1}{K} \right).
\end{aligned}
\label{Deep-NN-E-10}
\end{equation}
From the definitions of \(\widetilde{\varepsilon}_{i}^{y}\),
\(\varepsilon_{i}^{y}\) and Young inequality, we have
\begin{equation}
\begin{aligned}
\widetilde{\varepsilon}_{i}^{y}
&=\inf\limits_{\mathcal{NN}_{y,i,1+d, n}^{\varrho, L, \mathbf{n}}}
\mathbb{E}\left[\left|\widetilde{\mathcal{U}}_{i}^m
-\widetilde{u}(t_i,X^{\pi,m}_i)
+\widetilde{u}(t_i,X^{\pi,m}_i)
-\mathcal{NN}_{y,i,1+d, n}^{\varrho, L, \mathbf{n}}
(p_{i}^{m},\theta^{r_j}_{y,i})\right|^2\right]\\
&\le 2\mathbb{E}\left[\left|\mathcal{U}_i^{m} - \widetilde{\mathcal{U}}_{i}^m\right|^2\right]
+2\varepsilon_{i}^{y}.
\end{aligned}
\label{Deep-NN-E--YZK-21-h}
\end{equation}
By (\ref{Deep-NN-E--YZK--h1}), we restate (\ref{Deep-NN-E--YZK-21-h}) as
\begin{equation}
\widetilde{\varepsilon}_{i}^{y}
\le C\left(\frac{1}{K_i}
+\varepsilon_{i}^{y}\right)
\le C\left(\frac{1}{K}
+\varepsilon_{i}^{y}\right).
\label{Deep-NN-E--YZK-21-hh}
\end{equation}
Similarly, from the definitions of \(\widetilde{\varepsilon}_{i}^{z}\),
\(\varepsilon_{i}^{z}\), Young inequality and (\ref{Deep-NN-E--YZK--h1-z}), we have
\begin{equation}
\begin{aligned}
\widetilde{\varepsilon}_{i}^{z}
&=\inf\limits_{\mathcal{NN}_{z,i,1+d, n\times d}^{\varrho, L, \mathbf{n}}}
\mathbb{E}\left[\left|\widetilde{\mathcal{V}}_{i}^m
-\widetilde{v}(t_i,X^{\pi,m}_i)
+\widetilde{v}(t_i,X^{\pi,m}_i)
-\mathcal{NN}_{z,i,1+d, n\times d}^{\varrho, L, \mathbf{n}}
(p_i^m,\theta_{z,i}^{r_j})\right|^2\right]\\
&\le 2\mathbb{E}\left[\left|\mathcal{V}_{i}^m-\widetilde{\mathcal{V}}_{i}^m\right|^2\right]
+2\varepsilon_i^z\\
&\le  C\left(\frac{1}{Kh}
+\varepsilon_{i}^{z}\right).
\end{aligned}
\label{Deep-NN-E--YZK-21-hhz}
\end{equation}
Plugging (\ref{Deep-NN-E--YZK-21-hh}) and
(\ref{Deep-NN-E--YZK-21-hhz}) into (\ref{Deep-NN-E-10}), we deduce
\begin{equation}
\begin{aligned}
\mathbb{E}\left[\left|\mathcal{Y}_{i}-\mathcal{U}_{i}^m\right|^2\right]
\le&    C\left(\varepsilon_{i}^{y} +
h\varepsilon_{i}^{z} +\frac{1}{K} \right).
\end{aligned}
\label{Deep-NN-E-10h}
\end{equation}
Inserting (\ref{Deep-NN-E-10h}) into (\ref{Deep-NN-E-8}), we deduce
\begin{equation}
\begin{aligned}
\max\limits_{0\le i\le N-1}
\mathbb{E}\left[\left|Y_{t_i}-\mathcal{Y}_{i}\right|^2\right]
&\le  C\mathbb{E}\left[\left| g(X_T)-g(X_N^{\pi,m})\right|^2\right] + Ch\\
&\quad+C\mathbb{E}\left[\sum\limits_{i=0}^{N-1}
\int_{t_{i}}^{t_{i+1}} |Z_{s}-\bar{Z}_{t_i}|^2ds \right]\\
&\quad+C\sum\limits_{i=0}^{N-1}
\left[
N\varepsilon_{i}^{y} +
\varepsilon_{i}^{z} +\frac{N}{K}\right].
\end{aligned}
\label{Deep-NN-E-11}
\end{equation}

Step 3. Now, we prove the consistency of the $Z$-component. From the assumptions (ii)-(iii), Cauchy-Schwarz inequality and the tower property of conditional expectations, we have  (see \cite{GELC07,ZJ04})
\begin{equation}
\begin{aligned}
&\mathbb{E}\left[\int_{t_{i}}^{t_{i+1}}
|Z_{s}-\mathcal{V}^m_{i}|^2ds \right]\\
&\le \mathbb{E}\left[\int_{t_{i}}^{t_{i+1}}
|Z_{s}-\bar{Z}_{t_i}|^2ds \right]
+2dh\mathbb{E}\left[\int_{t_{i}}^{t_{i+1}}f_s^2ds\right]\\
&\quad+2d\mathbb{E}\left[
\left|
Y_{t_{i+1}}
-
\mathcal{NN}_{y,i+1,1+d,n}^{\varrho,L,\mathbf{n}}
(p^m_{i+1},\theta^{*}_{y,i+1})
\right|^2
\right]\\
&\quad-2d\mathbb{E}\left[
\left|
\mathbb{E}_{i}\left[
Y_{t_{i+1}}
-
\mathcal{NN}_{y,i+1,1+d,n}^{\varrho,L,\mathbf{n}}
(p^m_{i+1},\theta^{*}_{y,i+1})
\right]
\right|^2
\right]\\
&= \mathbb{E}\left[\int_{t_{i}}^{t_{i+1}}
|Z_{s}-\bar{Z}_{t_i}|^2ds \right]
+2dh\mathbb{E}\left[\int_{t_{i}}^{t_{i+1}}f_s^2ds\right]\\
&\quad+2d\mathbb{E}\left[
\left|
Y_{t_{i+1}}-\mathcal{Y}_{i+1}
\right|^2
\right]
-2d\mathbb{E}\left[
\left|
\mathbb{E}_{i}\left[
Y_{t_{i+1}}-\mathcal{Y}_{i+1}
\right]
\right|^2
\right]\\
&\le \mathbb{E}\left[\int_{t_{i}}^{t_{i+1}}
|Z_{s}-\bar{Z}_{t_i}|^2ds \right]
+2dh\mathbb{E}\left[\int_{t_{i}}^{t_{i+1}}f_s^2ds\right]\\
&\quad+4d\mathbb{E}\left[
\left|
Y_{t_{i+1}}-\mathcal{Y}_{i+1}
\right|^2
\right]
-4d\mathbb{E}\left[
\left|
\mathbb{E}_{i}\left[
Y_{t_{i+1}}-\mathcal{Y}_{i+1}
\right]
\right|^2
\right].
\end{aligned}
\label{Deep-NN-E-12}
\end{equation}
Summing over $i=0,1,\cdots,N-1$, we have
\begin{equation}
\begin{aligned}
&\mathbb{E}\left[\sum\limits_{i=0}^{N-1}\int_{t_{i}}^{t_{i+1}} |Z_{s}-\mathcal{V}^m_{i}|^2ds \right]\\
&\le \mathbb{E}\left[\sum\limits_{i=0}^{N-1}\int_{t_{i}}^{t_{i+1}} |Z_{s}-\bar{Z}_{t_i}|^2ds \right]+Ch+4d\mathbb{E}\left[\left| g(X_T)-g(X_N^{\pi,m})\right|^2\right]\\
&\quad+4d\mathbb{E}\left[\sum\limits_{i=0}^{N-1}\left| Y_{t_{i}}-\mathcal{Y}_{i}\right|^2\right]
-4d\mathbb{E}\left[\sum\limits_{i=0}^{N-1}\left|\mathbb{E}_{i}\left[ Y_{t_{i+1}}-\mathcal{Y}_{i+1}\right]\right|^2\right].
\label{Deep-NN-E-13}
\end{aligned}
\end{equation}
From Young inequality in the form: $(a+b)^2\le (1+\gamma h)a^2 + (1+\frac{1}{\gamma h})b^2,$ for $a,b \in \mathbb{R}$, and some $\gamma>0$ which will be given later, (\ref{Deep-NN-E-6}),
Cauchy-Schwarz inequality, and the Lipschitz condition with respect to $f$,
 we deduce
\begin{equation}
\begin{aligned}
&4d\mathbb{E}\left[\left| Y_{t_{i}}-\mathcal{Y}_{i}\right|^2\right]
-4d\mathbb{E}\left[\left|\mathbb{E}_{i}\left[ Y_{t_{i+1}}-\mathcal{Y}_{i+1}\right]\right|^2\right]\\
&\le4d\left(\frac{1+\gamma h}{1-h}\left(1+\frac{4L^2h}{\gamma}\right)-1\right)\mathbb{E}\left[\left|\mathbb{E}_{i}\left[ Y_{t_{i+1}}-\mathcal{Y}_{i+1}\right]\right|^2\right]\\
&\quad+\frac{16dL^2}{\gamma}\frac{1+\gamma h}{1-h} \left(Ch^2
+\mathbb{E}\left[\int_{t_{i}}^{t_{i+1}} |Z_{s}-\mathcal{V}^m_{i}|^2ds \right]\right)
+\frac{4d}{(1-h) h} \mathbb{E}\left[\left|\mathcal{Y}_{i}-\mathcal{U}^m_{i}\right|^2\right].
\label{Deep-NN-E-14}
\end{aligned}
\end{equation}
Taking $\gamma = 48dL^2$, we obtain $\frac{16dL^2}{\gamma}\frac{1+\gamma h}{1-h}\le \frac{1}{2}$ for sufficiently small $h$.
Then plugging (\ref{Deep-NN-E-14}) into (\ref{Deep-NN-E-13}), we have
\begin{equation}
\begin{aligned}
&\frac{1}{2}\mathbb{E}\left[\sum\limits_{i=0}^{N-1}\int_{t_{i}}^{t_{i+1}} |Z_{s}-\mathcal{V}^m_{i}|^2ds \right]\\
&\le \mathbb{E}\left[\sum\limits_{i=0}^{N-1}\int_{t_{i}}^{t_{i+1}} |Z_{s}-\bar{Z}_{t_i}|^2ds \right]+Ch+C\max\limits_{0\le i\le N}\mathbb{E}\left[\left| Y_{t_{i}}-\mathcal{Y}_{i}\right|^2\right]
+C N\sum\limits_{i=0}^{N-1}\mathbb{E}\left[\left|\mathcal{Y}_{i}-\mathcal{U}^m_{i}\right|^2\right]\\
&\le C\mathbb{E}\left[\left| g(X_T)-g(X_N^{\pi,m})\right|^2\right] + Ch
+C\mathbb{E}\left[\sum\limits_{i=0}^{N-1}\int_{t_{i}}^{t_{i+1}} |Z_{s}-\bar{Z}_{t_i}|^2ds \right]\\
&\quad+C\sum\limits_{i=0}^{N-1}\mathbb{E}\left[
N\varepsilon_{i}^{y} +
\varepsilon_{i}^{z} +\frac{N}{K}\right],
\label{Deep-NN-E-15}
\end{aligned}
\end{equation}
where we use \eqref{Deep-NN-E-10h} and \eqref{Deep-NN-E-11} in the last inequality.
From Young inequality, we know
\begin{equation}
\begin{aligned}
\mathbb{E}\left[\int_{t_{i}}^{t_{i+1}} |Z_{s}-\mathcal{Z}_{i}|^2ds \right]
\le 2\mathbb{E}\left[\int_{t_{i}}^{t_{i+1}} |Z_{s}-\mathcal{V}^m_{i}|^2ds \right]+2h\mathbb{E}\left[\left| \mathcal{V}^m_{i}-\mathcal{Z}_{i}\right|^2\right].
\label{Deep-NN-E-16}
\end{aligned}
\end{equation}
By Young inequality, we get
\begin{equation}
\begin{aligned}
\mathbb{E}\left[\left|\mathcal{Z}_{i}-\mathcal{V}^m_{i}\right|^2\right]
\le&  2\mathbb{E}\left[\left|\mathcal{Z}_{i}-\widetilde{\mathcal{V}}^m_{i}\right|^2\right]
+2\mathbb{E}\left[\left|\widetilde{\mathcal{V}}^m_{i}-\mathcal{V}^m_{i}\right|^2\right].
\label{Deep-NN-E-9-z}
\end{aligned}
\end{equation}
From (\ref{NS-MC-3--Y7}) and (\ref{Deep-NN-E--YZK--h1-z}),
we rewrite (\ref{Deep-NN-E-9-z}) as
\begin{equation}
\begin{aligned}
\mathbb{E}\left[\left|\mathcal{Z}_{i}-\mathcal{V}^m_{i}\right|^2\right]
\le&  C\left(\frac{1}{h}\widetilde{\varepsilon}_{i}^{y}
+\widetilde{\varepsilon}_{i}^{z}+\frac{1}{K_i h}\right)\\
\le&  C\left(N\widetilde{\varepsilon}_{i}^{y}
+\widetilde{\varepsilon}_{i}^{z}+\frac{N}{K_i}\right),
\end{aligned}
\label{Deep-NN-E-9-z1h}
\end{equation}
where we used \(h=T/N\) and absorbed the constant \(T\) into \(C\).
Inserting  (\ref{Deep-NN-E-9-z1h}) into (\ref{Deep-NN-E-16}), we deduce
\begin{equation}
\begin{aligned}
\mathbb{E}\left[\int_{t_{i}}^{t_{i+1}} |Z_{s}-\mathcal{Z}_{i}|^2ds \right]
\le& 2\mathbb{E}\left[\int_{t_{i}}^{t_{i+1}}
|Z_{s}-\mathcal{V}^m_{i}|^2ds \right]\\
&+C\left(\widetilde{\varepsilon}_{i}^{y}
+h\widetilde{\varepsilon}_{i}^{z}+\frac{1}{K_i}\right).
\end{aligned}
\label{Deep-NN-E-16-1}
\end{equation}
By (\ref{Deep-NN-E--YZK-21-hh}) and  (\ref{Deep-NN-E--YZK-21-hhz}), we rewrite
(\ref{Deep-NN-E-16-1}) as
\begin{equation}
\begin{aligned}
\mathbb{E}\left[\int_{t_{i}}^{t_{i+1}} |Z_{s}-\mathcal{Z}_{i}|^2ds \right]
\le& 2\mathbb{E}\left[\int_{t_{i}}^{t_{i+1}}
|Z_{s}-\mathcal{V}^m_{i}|^2ds \right]\\
&+C\left(\varepsilon_{i}^{y}
+h\varepsilon_{i}^{z}+\frac{1}{K_i}\right).
\end{aligned}
\label{Deep-NN-E-16-1h}
\end{equation}
Summing over \(i=0,1,\cdots,N-1\),  and then inserting
(\ref{Deep-NN-E-15}) into the derived inequality, we obtain
\begin{equation}
\begin{aligned}
&\mathbb{E}\left[\sum\limits_{i=0}^{N-1}
\int_{t_{i}}^{t_{i+1}} |Z_{s}-\mathcal{Z}_{i}|^2ds \right]\\
\le &  C\mathbb{E}\left[\left| g(X_T)-g(X_N^{\pi,m})\right|^2\right] + Ch
+C\mathbb{E}\left[\sum\limits_{i=0}^{N-1}
\int_{t_{i}}^{t_{i+1}} |Z_{s}-\bar{Z}_{t_i}|^2ds \right]\\
&+C\sum\limits_{i=0}^{N-1}
\left[
N\varepsilon_{i}^{y} +
\varepsilon_{i}^{z} +\frac{N}{K}\right].
\end{aligned}
\label{Deep-NN-E-16-1-1}
\end{equation}
Adding (\ref{Deep-NN-E-16-1-1}) to (\ref{Deep-NN-E-11}), we derive the required
error estimate and this ends the proof.
\end{proof}

\begin{remark}
Theorem \ref{ae-DMCE}  offers a structural advantage over Theorem 4.1 in Hure, Pham, and Warin \cite{HCPHWX20} through a finer decomposition of the error terms on the right-hand side of the inequality. While Theorem 4.1 bounds the approximation error by how well a neural network fits inherently noisy pathwise labels (\( \varepsilon_i^{\mathcal{N},v} \) and \( \varepsilon_i^{\mathcal{N},z} \)), Theorem 3.1 explicitly isolates the error into two categories: network fitting residuals ($\varepsilon_{i}^{y}$ and $\varepsilon_{i}^{z}$) and statistical Monte Carlo noise ($\frac{1}{K}$). This separation provides a theoretical explanation for the variance-reduction
mechanism underlying the DBR method. By taking conditional expectations, the stochastic perturbation
$\Delta W_i$ is smoothed out prior to loss evaluation. Consequently, the DBR scheme effectively recasts a variance-prone stochastic projection problem into a deterministic regression task, theoretically predicting smoother optimization landscapes, reduced fitting difficulty, and superior generalization capabilities compared to the DBDP method.
\end{remark}

\subsection{Convergence analysis}

In this section, we investigate the convergence rate of the approximation error for the proposed DBR method utilizing GroupSort deep neural networks  (see \cite{ACLJGR19,GMPHWX22}). To the best of our knowledge, Germain, Pham, and Warin \cite{GMPHWX22} were the first to employ GroupSort deep neural networks to analyze the convergence rates of deep learning methods for PDEs. Building upon their work, we adapt the analytical techniques introduced in Proposition 3.6 of \cite{GMPHWX22} to establish Theorem 3.4. Although the core idea closely parallels theirs, we derive a tailored result specific to our DBR framework. Before proceeding, we briefly review the GroupSort deep neural networks architecture.

The set of the GroupSort neural networks is defined as
$$
\begin{aligned}
 \mathcal{G}_{\mathcal{K}, \mathbf{d}_0, \mathbf{d}}^{\zeta_\kappa, L, \mathbf{n}}:=&\left\{\Psi=\left(\Psi_i\right)_{i=1,2, \ldots, \mathbf{d}}: \mathbb{R}^{\mathbf{d}_0} \mapsto \mathbb{R}^{\mathbf{d}}, \Psi_i (x) = \mathcal{K} \overline{\beta}_i \psi_i\left(\frac{x+\alpha_i}{\overline{\beta}_i}\right),\right. \\
&\left. \quad \psi_i \in \mathcal{S}_{\mathbf{d}_0}^{\zeta_\kappa,L, \mathbf{n}} \text { for some } \alpha_i \in \mathbb{R}^{\mathbf{d}_0}, \overline{\beta}_i>0\right\} .
\end{aligned}
$$
with
$\kappa \in \mathbb{N}^*, \kappa \geq 2$, be a grouping size, dividing the number of neurons $n_\ell=\kappa \tilde{n}_\ell$, at each layer $\ell=0, \ldots, L-1$;
$\sum\limits_{\ell=0}^{L-1} n_\ell$ denotes the width of the network and $L+1$ is the depth; $\zeta_\kappa=\left(\zeta_\kappa^\ell\right)_{\ell=0, \ldots, L-1}$ denotes a specific sequence of activation functions;
each nonlinear function $\zeta_\kappa^\ell$ divides its input into groups of size $\kappa$ and sorts each group in decreasing order; furthermore, by enforcing the parameters of the networks to satisfy, with the Euclidean norm $|\cdot|_2$ and the $L_{\infty}$ norm $|\cdot|_{\infty}$,
$$
\begin{aligned}
\mathcal{S}_{\mathbf{d}_0}^{\zeta_\kappa,L, \mathbf{n}}= & \left\{\varphi\left(\mathcal{W}_0, \beta_0, \ldots, \mathcal{W}_{L}, \beta_{L}\right) \in \mathcal{NN}_{\mathbf{d}_0, 1}^{\zeta_\kappa, L, \mathbf{n}}, \sup _{|x|_2=1}\left|\mathcal{W}_0 x\right|_{\infty} \leq 1, \sup _{|x|_{\infty}=1}\left|\mathcal{W}_i x\right|_{\infty} \leq 1\right. \\
& \left.\quad\left|\beta_j\right|_{\infty} \leq M, i=1, \ldots,L, j=0, \ldots, L\right\}.
\end{aligned}
$$

In what follows, we show the quantitative approximation result which is vital to study the convergence rates of the proposed deep probabilistic numerical method.

\begin{lemma} (see Proposition 2.1 in \cite{GMPHWX22})\label{gropusort-est}
Let $\tilde{f}:[-R,R]^d\rightarrow \mathbb{R}^{d_1}$ be $\mathcal{K}$-Lipschitz. Then, $\forall \varepsilon>0$, there exists a GroupSort neural network $g$ in $\mathcal{G}_{K,d,d_1}^{\zeta_{\kappa},L,\mathbf{n}}$ satisfying
$$
\sup\limits_{x\in[-R,R]^d}|\tilde{f}(x)-g(x)|\le 2\sqrt{d_1}R\mathcal{K}\varepsilon,
$$
with $g$ of grouping size $\kappa  = \lceil\frac{2\sqrt{d}}{\varepsilon}\rceil$, depth $L+1 = O(d^2)$ and width
$\sum\limits_{\ell=0}^{L-1}n_\ell=O((\frac{2\sqrt{d}}{\varepsilon})^{d^2-1})$ in the case $d>1$.
If $d=1$, the same result holds with $g$ of grouping size $\kappa  = \lceil\frac{1}{\varepsilon}\rceil$, depth $L+1 = 3$ and width
$\sum\limits_{\ell=0}^{L-1}n_\ell=O(\frac{1}{\varepsilon})$; the notation $\lceil x\rceil$ denotes the unique integer $n$ satisfying
the constraint $x\le n<x+1$.
\end{lemma}

\begin{theorem}
\label{Thm-GS-convergence}
Suppose that the assumptions (i)--(iv) hold. Furthermore, assume that
\(X_0\in L^{2+\delta}(\mathcal{F}_0,\mathbb{R}^d)\), for some \(\delta>0\), and that
\(g\) is an \(L_g\)-Lipschitz function. Then, there exists a bounded sequence
\(\{\mathcal{K}_i\}_{i=0}^{N}\), uniformly in \(i\) and \(N\), such that the spatial
slices of the GroupSort neural networks satisfy
\[
x\longmapsto
\mathcal{NN}_{y,i,1+d,n}^{\varrho,L,\mathbf{n}}
((t_i,x),\theta_{y,i}^{*})
\in
G_{\mathcal{K}_i,d,n}^{\zeta_\kappa,L,\mathbf{n}},
\]
and
\[
x\longmapsto
\mathcal{NN}_{z,i,1+d,n\times d}^{\varrho,L,\mathbf{n}}
((t_i,x),\theta_{z,i}^{*})
\in
G_{\sqrt{\frac{d}{h}}\mathcal{K}_i,d,n\times d}^{\zeta_\kappa,L,\mathbf{n}}.
\]
Let
\[
K_{\mathrm{MC}}:=\min_{0\leq i\leq N-1}K_i
\]
denote the minimal number of inner Monte Carlo samples used to approximate the
conditional expectations. If
\[
K_{\mathrm{MC}}\asymp N^3,
\]
then the DBR approximation \((\mathcal{Y}_i,\mathcal{Z}_i)_{0\leq i\leq N-1}\)
satisfies
\[
\max_{0\leq i\leq N}
\mathbb{E}\left[
\left|
Y_{t_i}-\mathcal{Y}_i
\right|^2
\right]
+
\mathbb{E}\left[
\sum_{i=0}^{N-1}
\int_{t_i}^{t_{i+1}}
\left|
Z_s-\mathcal{Z}_i
\right|^2ds
\right]
\leq Ch,
\]
where \(C\) is independent of \(N\), \(h\), and \(K_{\mathrm{MC}}\).

More precisely, the above rate is obtained by choosing the truncation radius
\(R\), the GroupSort approximation parameter \(\eta\), and the inner Monte Carlo
sample size \(K_{\mathrm{MC}}\) as
\[
R=O\left(N^{\frac{3}{\delta}}\right),
\qquad
\eta=O\left(N^{-\frac32-\frac{3}{\delta}}\right),
\qquad
K_{\mathrm{MC}}\asymp N^3.
\]
Consequently, in the case \(d>1\), the GroupSort networks can be chosen with
grouping size
\[
\kappa
=
O\left(
\left\lceil
2\sqrt d\,N^{\frac32+\frac{3}{\delta}}
\right\rceil
\right),
\]
depth
\[
L+1=O(d^2),
\]
and width
\[
\sum_{\ell=0}^{L-1}n_\ell
=
O\left(
\left(
2\sqrt d\,N^{\frac32+\frac{3}{\delta}}
\right)^{d^2-1}
\right).
\]
If \(d=1\), one can take
\[
\kappa
=
O\left(
N^{\frac32+\frac{3}{\delta}}
\right),
\qquad
L+1=3,
\qquad
\sum_{\ell=0}^{L-1}n_\ell
=
O\left(
N^{\frac32+\frac{3}{\delta}}
\right).
\]
\end{theorem}

\begin{proof} 
Step 1. We first prove the Lipschitz propagation of the target functions which will
be approximated by GroupSort neural networks. For notational simplicity, for every
fixed time level \(t_i\), we write
\[
p_i^x:=(t_i,x),\qquad x\in\mathbb{R}^d.
\]
Set
\[
\Psi_N(x):=g(x),
\]
and, for \(0\leq i\leq N-1\),
\[
\Psi_i(x):=
\mathcal{NN}_{y,i,1+d,n}^{\varrho,L,\mathbf{n}}
(p_i^x,\theta_{y,i}^{*}).
\]
For \(x\in\mathbb{R}^d\), define the one-step Euler process starting from \(x\) at
time \(t_i\) by
\begin{equation}
\begin{aligned}
X_{i+1}^{i,x}
:=
x+\mu(t_i,x)h+\sigma(t_i,x)\Delta W_i .
\end{aligned}
\label{GS-Xix}
\end{equation}
Then the deterministic functions \(y_i:\mathbb{R}^d\to\mathbb{R}^n\) and
\(z_i:\mathbb{R}^d\to\mathbb{R}^{n\times d}\) are defined recursively by
\begin{equation}
\begin{aligned}
y_i(x)
&:=
\mathbb{E}\left[
\Psi_{i+1}\left(X_{i+1}^{i,x}\right)
\right]
+h f\left(t_i,x,y_i(x),z_i(x)\right),\\
z_i(x)
&:=
\mathbb{E}\left[
\Psi_{i+1}\left(X_{i+1}^{i,x}\right)
\frac{\Delta W_i^\top}{h}
\right],
\end{aligned}
\label{GS-yz-def}
\end{equation}
where the expectation is taken with respect to the Brownian increment
\(\Delta W_i\). In particular, by the Markov property of the Euler scheme,
\[
\mathcal{U}_i^m=y_i(X_i^{\pi,m}),
\qquad
\mathcal{V}_i^m=z_i(X_i^{\pi,m}).
\]

Let \(x,x'\in\mathbb{R}^d\), and let \(X_{i+1}^{i,x}\) and
\(X_{i+1}^{i,x'}\) be driven by the same Brownian increment \(\Delta W_i\). From
the Lipschitz continuity of \(\mu\) and \(\sigma\), we have
\begin{equation}
\begin{aligned}
\mathbb{E}\left[
\left|
X_{i+1}^{i,x}-X_{i+1}^{i,x'}
\right|^2
\right]
\leq (1+Ch)|x-x'|^2 .
\end{aligned}
\label{GS-X-stability}
\end{equation}
Assume that \(\Psi_{i+1}\) is \(\mathcal{K}_{i+1}\)-Lipschitz with respect to the
spatial variable. Define
\[
\Delta\Psi_{i+1}
:=
\Psi_{i+1}\left(X_{i+1}^{i,x}\right)
-
\Psi_{i+1}\left(X_{i+1}^{i,x'}\right),
\qquad
\overline{\Delta\Psi}_{i+1}
:=
\mathbb{E}\left[\Delta\Psi_{i+1}\right].
\]
Then, by \eqref{GS-X-stability},
\begin{equation}
\begin{aligned}
\mathbb{E}\left[
\left|
\Delta\Psi_{i+1}
\right|^2
\right]
\leq
\mathcal{K}_{i+1}^2
\mathbb{E}\left[
\left|
X_{i+1}^{i,x}-X_{i+1}^{i,x'}
\right|^2
\right]
\leq
(1+Ch)\mathcal{K}_{i+1}^2|x-x'|^2 .
\end{aligned}
\label{GS-Psi-stability}
\end{equation}

We next estimate the Lipschitz constant of \(z_i\). Since
\(\mathbb{E}[\Delta W_i]=0\), we have
\begin{equation}
\begin{aligned}
h\left(z_i(x)-z_i(x')\right)
=
\mathbb{E}\left[
\left(
\Delta\Psi_{i+1}
-
\overline{\Delta\Psi}_{i+1}
\right)
\Delta W_i^\top
\right].
\end{aligned}
\label{GS-z-centered}
\end{equation}
Therefore, by Cauchy-Schwarz inequality,
\begin{equation}
\begin{aligned}
h\left|z_i(x)-z_i(x')\right|^2
&\leq
d\,
\mathbb{E}\left[
\left|
\Delta\Psi_{i+1}
-
\overline{\Delta\Psi}_{i+1}
\right|^2
\right] \\
&\leq
d\,\mathbb{E}\left[
\left|
\Delta\Psi_{i+1}
\right|^2
\right]\\
&\leq
d(1+Ch)\mathcal{K}_{i+1}^2|x-x'|^2 .
\end{aligned}
\label{GS-z-Lip}
\end{equation}

Now we estimate the Lipschitz constant of \(y_i\). From \eqref{GS-yz-def},
\begin{equation}
\begin{aligned}
y_i(x)-y_i(x')
=
\overline{\Delta\Psi}_{i+1}
+
h\Delta f_i,
\end{aligned}
\label{GS-y-diff}
\end{equation}
where
\[
\Delta f_i
:=
f(t_i,x,y_i(x),z_i(x))
-
f(t_i,x',y_i(x'),z_i(x')).
\]
By Young inequality,
\[
|a+b|^2
\leq
(1+\gamma h)|a|^2
+
\left(1+\frac{1}{\gamma h}\right)|b|^2,
\]
and by the Lipschitz continuity of \(f\), we obtain
\begin{equation}
\begin{aligned}
|y_i(x)-y_i(x')|^2
&\leq
(1+\gamma h)\left|\overline{\Delta\Psi}_{i+1}\right|^2\\
&\quad+
3L^2h^2\left(1+\frac{1}{\gamma h}\right)
\left(
|x-x'|^2
+
|y_i(x)-y_i(x')|^2
+
|z_i(x)-z_i(x')|^2
\right).
\end{aligned}
\label{GS-y-Lip-1}
\end{equation}
Using \eqref{GS-z-centered}--\eqref{GS-z-Lip}, we further get
\begin{equation}
\begin{aligned}
|y_i(x)-y_i(x')|^2
&\leq
\left(1+\gamma h-3dL^2\left(h+\frac{1}{\gamma}\right)\right)
\left|
\overline{\Delta\Psi}_{i+1}
\right|^2\\
&\quad
+3dL^2\left(h+\frac{1}{\gamma}\right)
\mathbb{E}\left[
\left|
\Delta\Psi_{i+1}
\right|^2
\right]\\
&\quad
+3L^2h\left(h+\frac{1}{\gamma}\right)
\left(
|x-x'|^2
+
|y_i(x)-y_i(x')|^2
\right).
\end{aligned}
\label{GS-y-Lip-2}
\end{equation}
Since
\[
\left|
\overline{\Delta\Psi}_{i+1}
\right|^2
\leq
\mathbb{E}\left[
\left|
\Delta\Psi_{i+1}
\right|^2
\right],
\]
it follows from \eqref{GS-y-Lip-2}, for \(h\) small enough, that
\begin{equation}
\begin{aligned}
|y_i(x)-y_i(x')|^2
&\leq
(1+Ch)
\mathbb{E}\left[
\left|
\Delta\Psi_{i+1}
\right|^2
\right]
+
Ch|x-x'|^2\\
&\leq
\left((1+Ch)\mathcal{K}_{i+1}^2+Ch\right)|x-x'|^2 .
\end{aligned}
\label{GS-y-Lip-3}
\end{equation}

Set \(\mathcal{K}_N=L_g\) and define recursively
\begin{equation}
\begin{aligned}
\mathcal{K}_i^2
:=
(1+Ch)\mathcal{K}_{i+1}^2+Ch,
\qquad i=N-1,N-2,\ldots,0.
\end{aligned}
\label{GS-K-recursion}
\end{equation}
By the discrete Gronwall inequality, the sequence
\(\{\mathcal{K}_i\}_{i=0}^{N}\) is bounded uniformly in \(i\) and \(N\). Hence,
\begin{equation}
\begin{aligned}
|y_i(x)-y_i(x')|
\leq
\mathcal{K}_i |x-x'|.
\end{aligned}
\label{GS-y-final-Lip}
\end{equation}
Moreover, since \(\mathcal{K}_i^2\geq (1+Ch)\mathcal{K}_{i+1}^2\), from
\eqref{GS-z-Lip} we also have
\begin{equation}
\begin{aligned}
\sqrt{h}\,|z_i(x)-z_i(x')|
\leq
\sqrt{d}\,\mathcal{K}_i|x-x'|.
\end{aligned}
\label{GS-z-final-Lip}
\end{equation}
Thus, \(y_i\) is \(\mathcal{K}_i\)-Lipschitz and \(z_i\) is
\(\sqrt{d/h}\mathcal{K}_i\)-Lipschitz with respect to the spatial variable,
uniformly in \(i\) and \(N\).

Step 2. We approximate the spatial function \(y_i\) by means of a
\(\mathcal{K}_i\)-Lipschitz GroupSort neural network
\[
\mathcal{NN}_{y,i,1+d,n}^{\varrho,L,\mathbf{n}}(p_i^x,\theta_{y,i}^{*})
\in
G_{\mathcal{K}_i,1+d,n}^{\zeta_\kappa,L,\mathbf{n}}.
\]
Since \(t_i\) is fixed at each time step, Lemma \ref{gropusort-est} is applied to the spatial
variable \(x\in\mathbb{R}^d\). Hence, on \([-R,R]^d\), there exists such a
network satisfying
\[
\sup_{x\in[-R,R]^d}
\left|
y_i(x)-
\mathcal{NN}_{y,i,1+d,n}^{\varrho,L,\mathbf{n}}(p_i^x,\theta_{y,i}^{*})
\right|
\leq
2\mathcal{K}_i R\varepsilon .
\]
Similarly, by \eqref{GS-z-final-Lip}, \(z_i\) can be approximated by a
\(\sqrt{d/h}\mathcal{K}_i\)-Lipschitz GroupSort neural network
\[
\mathcal{NN}_{z,i,1+d,n\times d}^{\varrho,L,\mathbf{n}}(p_i^x,\theta_{z,i}^{*})
\in
G_{\sqrt{d/h}\mathcal{K}_i,1+d,n\times d}^{\zeta_\kappa,L,\mathbf{n}},
\]
and we have
\[
\sup_{x\in[-R,R]^d}
\sqrt{h}\left|
z_i(x)-
\mathcal{NN}_{z,i,1+d,n\times d}^{\varrho,L,\mathbf{n}}(p_i^x,\theta_{z,i}^{*})
\right|
\leq
2\sqrt d\,\mathcal{K}_i R\varepsilon .
\]

Step 3. We now estimate the approximation errors on the whole space. Let
\[
q:=1+\frac{\delta}{2},
\qquad
2q=2+\delta.
\]
For notational simplicity, write
\[
\mathcal{N}_{y,i}(x;\theta_{y,i})
:=
\mathcal{NN}_{y,i,1+d,n}^{\varrho,L,\mathbf{n}}
(p_i^x,\theta_{y,i}),
\qquad
\mathcal{N}_{z,i}(x;\theta_{z,i})
:=
\mathcal{NN}_{z,i,1+d,n\times d}^{\varrho,L,\mathbf{n}}
(p_i^x,\theta_{z,i}),
\]
where \(p_i^x=(t_i,x)\). Let \(\mathcal{N}_{y,i}^{*}\) and
\(\mathcal{N}_{z,i}^{*}\) be the GroupSort networks constructed in Step 2. Since
\(t_i\) is fixed at each time step, the GroupSort approximation result is applied
to the spatial variable \(x\in\mathbb{R}^d\). Hence, for \(R\geq 1\), there exists
a GroupSort network \(\mathcal{N}_{y,i}^{*}\) such that
\begin{equation}
\begin{aligned}
\sup_{x\in[-R,R]^d}
\left|
y_i(x)-\mathcal{N}_{y,i}^{*}(x)
\right|
\leq C R\eta,
\end{aligned}
\label{GS-y-uniform-error}
\end{equation}
where \(\eta>0\) is the approximation parameter in Lemma \ref{gropusort-est}.
Similarly, there exists a GroupSort network \(\mathcal{N}_{z,i}^{*}\) such that
\begin{equation}
\begin{aligned}
\sup_{x\in[-R,R]^d}
\sqrt{h}\left|
z_i(x)-\mathcal{N}_{z,i}^{*}(x)
\right|
\leq C R\eta.
\end{aligned}
\label{GS-z-uniform-error}
\end{equation}
Here and below, the constant \(C\) may depend on \(d,n,T,L\), and on the uniform
bound of the sequence \(\{\mathcal{K}_i\}_{i=0}^{N}\), but is independent of
\(N,R,\eta\), and the Monte Carlo sample sizes.

We first estimate the approximation error of the \(Y\)-component. By splitting the
domain into \(\{|X_i^{\pi,m}|\leq R\}\) and \(\{|X_i^{\pi,m}|>R\}\), we have
\begin{equation}
\begin{aligned}
&\inf_{\mathcal{NN}_{y,i}\in
\mathcal{G}_{\mathcal{K}_i,1+d,n}^{\zeta_\kappa,L,\mathbf{n}}}
\mathbb{E}\left[
\left|
y_i(X_i^{\pi,m})
-
\mathcal{NN}_{y,i,1+d,n}^{\varrho,L,\mathbf{n}}
(p_i^m,\theta_{y,i})
\right|^2
\right]  \\
&\leq
\mathbb{E}\left[
\left|
y_i(X_i^{\pi,m})
-
\mathcal{N}_{y,i}^{*}(X_i^{\pi,m})
\right|^2
\mathbf{1}_{\{|X_i^{\pi,m}|\leq R\}}
\right]  \\
&\quad+
\mathbb{E}\left[
\left|
y_i(X_i^{\pi,m})
-
\mathcal{N}_{y,i}^{*}(X_i^{\pi,m})
\right|^2
\mathbf{1}_{\{|X_i^{\pi,m}|> R\}}
\right].
\end{aligned}
\label{GS-y-split}
\end{equation}
From \eqref{GS-y-uniform-error}, the first term on the right-hand side of
\eqref{GS-y-split} is bounded by
\begin{equation}
\begin{aligned}
\mathbb{E}\left[
\left|
y_i(X_i^{\pi,m})
-
\mathcal{N}_{y,i}^{*}(X_i^{\pi,m})
\right|^2
\mathbf{1}_{\{|X_i^{\pi,m}|\leq R\}}
\right]
\leq C R^2\eta^2.
\end{aligned}
\label{GS-y-local}
\end{equation}
For the tail term, by H\"older inequality,
\begin{equation}
\begin{aligned}
&\mathbb{E}\left[
\left|
y_i(X_i^{\pi,m})
-
\mathcal{N}_{y,i}^{*}(X_i^{\pi,m})
\right|^2
\mathbf{1}_{\{|X_i^{\pi,m}|> R\}}
\right]\\
&\leq
\left(
\mathbb{E}\left[
\left|
y_i(X_i^{\pi,m})
-
\mathcal{N}_{y,i}^{*}(X_i^{\pi,m})
\right|^{2q}
\right]
\right)^{\frac1q}
\left(
\mathbb{P}\left(|X_i^{\pi,m}|>R\right)
\right)^{\frac{q-1}{q}} .
\end{aligned}
\label{GS-y-tail-holder}
\end{equation}
Since both \(y_i\) and \(\mathcal{N}_{y,i}^{*}\) are Lipschitz uniformly in \(i\) and
\(N\), and since
\[
\left|
y_i(0)-\mathcal{N}_{y,i}^{*}(0)
\right|
\leq C R\eta,
\]
we have
\begin{equation}
\begin{aligned}
\left(
\mathbb{E}\left[
\left|
y_i(X_i^{\pi,m})
-
\mathcal{N}_{y,i}^{*}(X_i^{\pi,m})
\right|^{2q}
\right]
\right)^{\frac1q}
\leq C\left(1+R^2\eta^2\right).
\end{aligned}
\label{GS-y-L2q-bound}
\end{equation}
Moreover, by Markov inequality and the standard moment estimate for the Euler
scheme,
\begin{equation}
\begin{aligned}
\mathbb{P}\left(|X_i^{\pi,m}|>R\right)
\leq
\frac{\mathbb{E}\left[|X_i^{\pi,m}|^{2q}\right]}{R^{2q}}
\leq
\frac{C}{R^{2q}},
\end{aligned}
\label{GS-X-tail}
\end{equation}
where we used \(X_0\in L^{2+\delta}\) and \(2q=2+\delta\). Therefore,
\begin{equation}
\begin{aligned}
\left(
\mathbb{P}\left(|X_i^{\pi,m}|>R\right)
\right)^{\frac{q-1}{q}}
\leq
\frac{C}{R^{2(q-1)}}
=
\frac{C}{R^\delta}.
\end{aligned}
\label{GS-X-tail-delta}
\end{equation}
Combining \eqref{GS-y-tail-holder}, \eqref{GS-y-L2q-bound}, and
\eqref{GS-X-tail-delta}, we get
\begin{equation}
\begin{aligned}
\mathbb{E}\left[
\left|
y_i(X_i^{\pi,m})
-
\mathcal{N}_{y,i}^{*}(X_i^{\pi,m})
\right|^2
\mathbf{1}_{\{|X_i^{\pi,m}|> R\}}
\right]
\leq
C\left(1+R^2\eta^2\right)R^{-\delta}.
\end{aligned}
\label{GS-y-tail}
\end{equation}
Consequently, by \eqref{GS-y-split}, \eqref{GS-y-local}, and \eqref{GS-y-tail},
for \(R\geq 1\) and \(R\eta\leq 1\), we obtain
\begin{equation}
\begin{aligned}
\varepsilon_i^y
&\leq
C\left(
R^2\eta^2+R^{-\delta}
\right).
\end{aligned}
\label{Deep-NN-E-r-7}
\end{equation}

Next, we estimate the approximation error of the \(Z\)-component. Applying the same
argument to the functions \(\sqrt{h}\,z_i\) and
\(\sqrt{h}\,\mathcal{N}_{z,i}^{*}\), and using \eqref{GS-z-uniform-error}, we obtain
\begin{equation}
\begin{aligned}
h\varepsilon_i^z
&=
\inf_{\mathcal{NN}_{z,i}\in
\mathcal{G}_{\sqrt{\frac{d}{h}}\mathcal{K}_i,1+d,n\times d}^{\zeta_\kappa,L,\mathbf{n}}}
\mathbb{E}\left[
h\left|
z_i(X_i^{\pi,m})
-
\mathcal{NN}_{z,i,1+d,n\times d}^{\varrho,L,\mathbf{n}}
(p_i^m,\theta_{z,i})
\right|^2
\right]\\
&\leq
C\left(
R^2\eta^2+R^{-\delta}
\right).
\end{aligned}
\label{Deep-NN-E-r-8}
\end{equation}
Set
\begin{equation}
\begin{aligned}
a_{R,\eta}:=R^2\eta^2+R^{-\delta}.
\end{aligned}
\label{GS-a-R-eta}
\end{equation}
Then \eqref{Deep-NN-E-r-7} and \eqref{Deep-NN-E-r-8} imply
\begin{equation}
\begin{aligned}
\varepsilon_i^y\leq C a_{R,\eta},
\qquad
h\varepsilon_i^z\leq C a_{R,\eta}.
\end{aligned}
\label{GS-eps-yz}
\end{equation}
Therefore,
\begin{equation}
\begin{aligned}
\sum_{i=0}^{N-1}
\left(
N\varepsilon_i^y+\varepsilon_i^z
\right)
\leq
C N^2 a_{R,\eta}.
\end{aligned}
\label{GS-sum-eps}
\end{equation}
Let
\[
K_{\mathrm{MC}}:=\min_{0\leq i\leq N-1}K_i
\]
be the minimal number of inner Monte Carlo samples. To achieve the convergence
rate \(O(1/N)\) in Theorem \ref{ae-DMCE}, it is sufficient to choose
\(R,\eta,K_{\mathrm{MC}}\) such that
\begin{equation}
\begin{aligned}
N^2\left(R^2\eta^2+R^{-\delta}\right)
=
O\left(\frac1N\right),
\qquad
\frac{N^2}{K_{\mathrm{MC}}}
=
O\left(\frac1N\right).
\end{aligned}
\label{dmce-cr-1}
\end{equation}
This is verified by taking
\begin{equation}
\begin{aligned}
R=O\left(N^{\frac{3}{\delta}}\right),
\qquad
\eta=O\left(N^{-\frac32-\frac{3}{\delta}}\right),
\qquad
K_{\mathrm{MC}}=O(N^3).
\end{aligned}
\label{GS-parameter-choice}
\end{equation}
Indeed, under this choice,
\[
R^2\eta^2=O(N^{-3}),
\qquad
R^{-\delta}=O(N^{-3}),
\qquad
\frac{N^2}{K_{\mathrm{MC}}}=O(N^{-1}).
\]

From Lemma \ref{gropusort-est}, if \(d>1\), the GroupSort neural networks can be
chosen with grouping size
\begin{equation}
\begin{aligned}
\kappa
=
O\left(
\left\lceil
2\sqrt{d}\,N^{\frac32+\frac{3}{\delta}}
\right\rceil
\right),
\end{aligned}
\label{GS-kappa-general-delta}
\end{equation}
depth \(L+1=O(d^2)\), and width
\begin{equation}
\begin{aligned}
\sum_{\ell=0}^{L-1}n_\ell
=
O\left(
\left(
2\sqrt{d}\,N^{\frac32+\frac{3}{\delta}}
\right)^{d^2-1}
\right).
\end{aligned}
\label{GS-width-general-delta}
\end{equation}
If \(d=1\), one can take
\begin{equation}
\begin{aligned}
\kappa
=
O\left(
N^{\frac32+\frac{3}{\delta}}
\right),
\qquad
L+1=3,
\qquad
\sum_{\ell=0}^{L-1}n_\ell
=
O\left(
N^{\frac32+\frac{3}{\delta}}
\right).
\end{aligned}
\label{GS-d1-general-delta}
\end{equation}

Since \(g\) is Lipschitz and
\[
\max_{0\leq i\leq N-1}
\mathbb{E}\left[
\sup_{s\in[t_i,t_{i+1}]}
|X_s-X_i^{\pi,m}|^2
\right]\leq Ch
\]
see \cite{KPEPE92}, we obtain
\begin{equation}
\begin{aligned}
\mathbb{E}\left[
\left|
g(X_T)-g(X_N^{\pi,m})
\right|^2
\right]
+
\mathbb{E}\left[
\sum_{i=0}^{N-1}
\int_{t_i}^{t_{i+1}}
|Z_s-\bar{Z}_{t_i}|^2ds
\right]
\leq Ch.
\end{aligned}
\label{dmce-cr-2}
\end{equation}
Combining \eqref{Deep-NN-E-r-7}, \eqref{Deep-NN-E-r-8},
\eqref{dmce-cr-1}, and \eqref{dmce-cr-2} with Theorem \ref{ae-DMCE}, we finally
obtain
\begin{equation}
\begin{aligned}
\max_{0\leq i\leq N-1}
\mathbb{E}\left[
\left|
Y_{t_i}-\mathcal{Y}_i
\right|^2
\right]
+
\mathbb{E}\left[
\sum_{i=0}^{N-1}
\int_{t_i}^{t_{i+1}}
|Z_s-\mathcal{Z}_i|^2ds
\right]
\leq Ch,
\end{aligned}
\nonumber
\end{equation}
which completes the proof.
\end{proof}

\section{Extension to variational inequalities}
In this section, we extend the DBR scheme to variational inequalities, referring to the resulting method as the RDBR scheme.
Then we study the convergence of the RDBR scheme for the variational inequalities.
Since the obstacle constraint requires an order relation and a projection operator,
we restrict this section to the scalar case \(n=1\). Thus, in this section,
\(u,Y,g,\Phi\) are real-valued, \(Z\in\mathbb{R}^{1\times d}\), and the reflecting
process \(\mathbf{K}\) is a real-valued adapted non-decreasing process. The maximum
operator \(\max\{\cdot,\cdot\}\) is understood in the usual scalar sense. We keep the
boldface notation \(\mathbf{K}\) in this section to distinguish the reflecting process
from the Monte Carlo sample sizes \(K_i\). We also assume the compatibility condition
\[
g(x)\geq \Phi(x),\qquad x\in\mathbb{R}^d.
\]

\subsection{RDBR Scheme}

Consider a variational inequality of the following form:
\begin{equation}
\begin{cases}
\min \left(
-\partial_t u-\mu \cdot D_x u
-\frac{1}{2}\operatorname{Tr}\left(\sigma\sigma^{\top}D_x^2u\right)
-f\left(\cdot,\cdot,u,\sigma^{\top}D_xu\right),
u-\Phi
\right)=0,
& \text{on }[0,T)\times\mathbb{R}^d,\\
u(T,x)=g(x),
& \text{on }\mathbb{R}^d.
\end{cases}
\label{eq:VI}
\end{equation}
This inequality arises, for instance, in optimal stopping problems and American
option pricing. It is established that such a variational inequality is associated
with a reflected forward backward stochastic differential equation (RFBSDE) given by
\begin{equation}
\begin{cases}
X_t=x_0+\displaystyle\int_0^t\mu(s,X_s)ds
+\displaystyle\int_0^t\sigma(s,X_s)dW_s,\\[1mm]
Y_t=g(X_T)+\displaystyle\int_t^T f(s,X_s,Y_s,Z_s)ds
-\displaystyle\int_t^T Z_s dW_s+\mathbf{K}_T-\mathbf{K}_t,\\[1mm]
Y_t\geq \Phi(X_t), \quad \forall t\in[0,T],\\[1mm]
\displaystyle\int_0^T\left(Y_t-\Phi(X_t)\right)d\mathbf{K}_t=0,
\end{cases}
\label{eq:RBSDE}
\end{equation}
where $\Phi:\mathbb{R}^{d}\rightarrow \mathbb{R}$ denotes the reflecting barrier
and $\mathbf{K}$ is a real-valued adapted non-decreasing process satisfying the
Skorokhod condition (see \cite{ElKNKCPEPSQMC97}).

The Euler time-discretization of the RFBSDE \eqref{eq:RBSDE}, at the mesh points
\(t_i\), is
\begin{equation}
Y_i^\pi
=
Y_{i+1}^\pi
+
h f(t_i,X_i^\pi,Y_i^\pi,Z_i^\pi)
-
Z_i^\pi\Delta W_i
+
\mathbf{K}_{i+1}^\pi-\mathbf{K}_i^\pi.
\nonumber
\end{equation}
Equivalently, it can be written in the following conditional expectation form:
\begin{equation}
\left\{
\begin{aligned}
\overline{Y}_i^\pi
=&\ \mathbb{E}_i\left[
Y_{i+1}^\pi
+
h f(t_i,X_i^\pi,\overline{Y}_i^\pi,Z_i^\pi)
\right],\\
Y_i^\pi
=&\ \max\left\{
\overline{Y}_i^\pi,\Phi(X_i^\pi)
\right\},\\
Z_i^\pi
=&\ \mathbb{E}_i\left[
Y_{i+1}^\pi\frac{\Delta W_i^\top}{h}
\right].
\end{aligned}
\right.
\label{RDBR-Euler}
\end{equation}

The backward process \(Y_{\cdot}^{\pi,m}\) in \eqref{eq:RBSDE} is simulated by the
Monte Carlo method in the following way, for \(i=N-1,N-2,\ldots,0\):
\begin{equation}
Y_i^{\pi,m}
=
Y_{i+1}^{\pi,m}
+
h f(t_i,X_i^{\pi,m},Y_i^{\pi,m},Z_i^{\pi,m})
-
Z_i^{\pi,m}\Delta W_i^m
+
\mathbf{K}_{i+1}^{\pi,m}-\mathbf{K}_i^{\pi,m}.
\nonumber
\end{equation}
This can also be expressed as the conditional expectation form
\begin{equation}
\left\{
\begin{aligned}
\overline{Y}_i^{\pi,m}
=&\ \mathbb{E}_i\left[
Y_{i+1}^{\pi,m}
+
h f(t_i,X_i^{\pi,m},\overline{Y}_i^{\pi,m},Z_i^{\pi,m})
\right],\\
Y_i^{\pi,m}
=&\ \max\left\{
\overline{Y}_i^{\pi,m},
\Phi(X_i^{\pi,m})
\right\},\\
Z_i^{\pi,m}
=&\ \mathbb{E}_i\left[
Y_{i+1}^{\pi,m}
\frac{(\Delta W_i^m)^\top}{h}
\right].
\end{aligned}
\right.
\label{NS-M-R}
\end{equation}

Thus, we define the fully discrete approximations
\(\{\widehat{\mathcal{Y}}_i\}_{i=0}^{N}\) and
\(\{\widehat{\mathcal{Z}}_i\}_{i=0}^{N-1}\) for the RDBR scheme associated with
the solutions \((Y_{t_i},Z_{t_i})\) of the variational inequality \eqref{eq:VI}, for
\(m=1,2,\ldots,M\), as follows:
\begin{enumerate}
\item
the terminal condition is
\[
\widehat{\mathcal{Y}}_N=g(X_N^{\pi,m}).
\]
No terminal value \(\widehat{\mathcal{Z}}_N\) is required by the backward scheme.

\item
for \(0\leq i<N\), the transition from \(i+1\) to \(i\) is given by
\begin{equation}
\left\{
\begin{aligned}
\overline{\mathcal{Y}}_i
=&\
\widehat{\mathcal{NN}}_{y,i,1+d,n}^{\varrho,L,\mathbf{n}}
(p_i^m,\theta_{y,i}^{*}),\\
\widehat{\mathcal{Y}}_i
=&\
\max\left\{
\overline{\mathcal{Y}}_i,\Phi(X_i^{\pi,m})
\right\},\\
\widehat{\mathcal{Z}}_i
=&\
\widehat{\mathcal{NN}}_{z,i,1+d,n\times d}^{\varrho,L,\mathbf{n}}
(p_i^m,\theta_{z,i}^{*}).
\end{aligned}
\right.
\label{NS-DMC-R}
\end{equation}
\end{enumerate}

For the training at time \(t_i\), the next-step value used in the conditional
expectation is the post-projection value. For \(m=1,\ldots,M\) and
\(k=1,\ldots,K_i\), define
\begin{equation}
\begin{aligned}
\widehat{\mathcal{Y}}_{i+1}^{m,k}
:=
\begin{cases}
g(X_N^{\pi,m,k}),
& i=N-1,\\[2mm]
\max\left\{
\widehat{\mathcal{NN}}_{y,i+1,1+d,n}^{\varrho,L,\mathbf{n}}
(p_{i+1}^{m,k},\theta_{y,i+1}^{*}),
\Phi(X_{i+1}^{\pi,m,k})
\right\},
& 0\leq i\leq N-2.
\end{cases}
\end{aligned}
\label{RDBR-next-projected-value-MC}
\end{equation}
Then the empirical losses are given by
\begin{equation}
\left\{
\begin{aligned}
\widehat{F}_{z,i}^{r_j}(\theta_{z,i}^{r_j})
=&\
\frac{1}{M}\sum_{m=1}^M
\left|
\frac{1}{K_i}\sum_{k=1}^{K_i}
\widehat{\mathcal{Y}}_{i+1}^{m,k}
\frac{(\Delta W_i^{m,k})^\top}{h}
-
\widehat{\mathcal{NN}}_{z,i,1+d,n\times d}^{\varrho,L,\mathbf{n}}
(p_i^m,\theta_{z,i}^{r_j})
\right|^2,\\
\widehat{F}_{y,i}^{r_j}(\theta_{y,i}^{r_j})
=&\
\frac{1}{M}\sum_{m=1}^M
\Bigg|
\frac{1}{K_i}\sum_{k=1}^{K_i}
\widehat{\mathcal{Y}}_{i+1}^{m,k}
+
h f\Big(
t_i,X_i^{\pi,m},
\widehat{\mathcal{NN}}_{y,i,1+d,n}^{\varrho,L,\mathbf{n}}
(p_i^m,\theta_{y,i}^{r_j}),\\
&\qquad\qquad\qquad\qquad
\widehat{\mathcal{NN}}_{z,i,1+d,n\times d}^{\varrho,L,\mathbf{n}}
(p_i^m,\theta_{z,i}^{*})
\Big)
-
\widehat{\mathcal{NN}}_{y,i,1+d,n}^{\varrho,L,\mathbf{n}}
(p_i^m,\theta_{y,i}^{r_j})
\Bigg|^2.
\end{aligned}
\right.
\label{RDBR-loss}
\end{equation}
Here
\[
\theta_{z,i}^{*}
\in
\arg\min_{\theta_{z,i}^{r_j}}
\widehat{F}_{z,i}^{r_j}(\theta_{z,i}^{r_j}),
\qquad
\theta_{y,i}^{*}
\in
\arg\min_{\theta_{y,i}^{r_j}}
\widehat{F}_{y,i}^{r_j}(\theta_{y,i}^{r_j}).
\]

As in Section~3, the losses \(\widehat F_{y,i}^{r_j}\) and
\(\widehat F_{z,i}^{r_j}\) above are empirical losses used in the numerical
implementation. The convergence analysis below is carried out at the population-loss
level. That is, all approximation errors are measured with respect to the true
expectation \(\mathbb E[\cdot]\), and exact minimization of the corresponding
population losses is assumed. The additional errors caused by empirical risk
approximation and non-exact stochastic optimization are not included in
Theorem~\ref{ae-DMCE-R}.

\subsection{Convergence analysis}

In this section, we analyze the convergence of the RDBR scheme \eqref{NS-DMC-R}
for the variational inequality \eqref{eq:VI} related to the solution \((Y,Z)\) of the
RFBSDE \eqref{eq:RBSDE}.

For the reflected scheme, the next-step value used in the conditional expectation
for the \(Z\)-component is the post-projection value. Therefore, for
\(i=N-1,\ldots,0\), we set
\begin{equation}
\begin{aligned}
\widehat{\mathcal{Y}}_{i+1}^{m}
:=
\begin{cases}
g(X_N^{\pi,m}),
& i=N-1,\\[2mm]
\max\left\{
\widehat{\mathcal{NN}}_{y,i+1,1+d,n}^{\varrho,L,\mathbf{n}}
(p_{i+1}^{m},\theta_{y,i+1}^{*}),
\Phi(X_{i+1}^{\pi,m})
\right\},
& 0\leq i\leq N-2.
\end{cases}
\end{aligned}
\label{RDBR-next-projected-value}
\end{equation}

Now, we investigate the errors of the scheme \eqref{NS-DMC-R} and define the
auxiliary quantities as below, for \(i=N-1,N-2,\ldots,0\):
\begin{equation}
\left\{
\begin{aligned}
\mathcal{U}_i^m
=&\
\mathbb{E}_i\left[
\widehat{\mathcal{Y}}_{i+1}^{m}
+
h f(t_i,X_i^{\pi,m},\mathcal{U}_i^m,\widehat{\mathcal{V}}_i^m)
\right],\\
\widehat{\mathcal{U}}_i^m
=&\
\max\left\{
\mathcal{U}_i^m,\Phi(X_i^{\pi,m})
\right\},\\
\widehat{\mathcal{V}}_i^m
=&\
\mathbb{E}_i\left[
\widehat{\mathcal{Y}}_{i+1}^{m}
\frac{(\Delta W_i^m)^\top}{h}
\right].
\end{aligned}
\right.
\label{Deep-NN-E-R}
\end{equation}
By the Markov property of the discretized forward process
\(\{X_i^{\pi,m}\}_{0\leq i\leq N}\), we have
\begin{equation}
\begin{aligned}
\mathcal{U}_i^m=\overline{u}(t_i,X_i^{\pi,m}),
\qquad
\widehat{\mathcal{V}}_i^m=\widehat{v}(t_i,X_i^{\pi,m}),
\end{aligned}
\label{Deep-NN-E--YZK-1R}
\end{equation}
where \(\overline{u}\) and \(\widehat{v}\) are some deterministic but unknown
functions.

For the Monte Carlo approximation of these conditional expectations, define
\begin{equation}
\left\{
\begin{aligned}
\widetilde{\mathcal{U}}_i^m
:=&\
\frac{1}{K_i}
\sum_{k=1}^{K_i}
\widehat{\mathcal{Y}}_{i+1}^{m,k}
+
h f\left(
t_i,X_i^{\pi,m},
\widetilde{\mathcal{U}}_i^m,
\widetilde{\widehat{\mathcal{V}}}_i^m
\right),\\
\widetilde{\widehat{\mathcal{U}}}_i^m
:=&\
\max\left\{
\widetilde{\mathcal{U}}_i^m,\Phi(X_i^{\pi,m})
\right\},\\
\widetilde{\widehat{\mathcal{V}}}_i^m
:=&\
\frac{1}{K_i}
\sum_{k=1}^{K_i}
\widehat{\mathcal{Y}}_{i+1}^{m,k}
\frac{(\Delta W_i^{m,k})^\top}{h}.
\end{aligned}
\right.
\label{RDBR-AUX-MC}
\end{equation}

Let us introduce
\begin{equation}
\begin{aligned}
\widehat{\varepsilon}_{i}^{y}
=&\
\inf_{\widehat{\mathcal{NN}}_{y,i,1+d,n}^{\varrho,L,\mathbf{n}}}
\mathbb{E}\left[
\left|
\overline{u}(t_i,X_i^{\pi,m})
-
\widehat{\mathcal{NN}}_{y,i,1+d,n}^{\varrho,L,\mathbf{n}}
(p_i^m,\theta_{y,i}^{r_j})
\right|^2
\right],\\
\widehat{\varepsilon}_{i}^{z}
=&\
\inf_{\widehat{\mathcal{NN}}_{z,i,1+d,n\times d}^{\varrho,L,\mathbf{n}}}
\mathbb{E}\left[
\left|
\widehat{v}(t_i,X_i^{\pi,m})
-
\widehat{\mathcal{NN}}_{z,i,1+d,n\times d}^{\varrho,L,\mathbf{n}}
(p_i^m,\theta_{z,i}^{r_j})
\right|^2
\right],
\end{aligned}
\label{Deep-NN-E--YZK-21R}
\end{equation}
and
\begin{equation}
\begin{aligned}
\overline{\varepsilon}_{i}^{y}
=&\
\inf_{\widehat{\mathcal{NN}}_{y,i,1+d,n}^{\varrho,L,\mathbf{n}}}
\mathbb{E}\left[
\left|
\widetilde{\mathcal{U}}_i^m
-
\widehat{\mathcal{NN}}_{y,i,1+d,n}^{\varrho,L,\mathbf{n}}
(p_i^m,\theta_{y,i}^{r_j})
\right|^2
\right],\\
\overline{\varepsilon}_{i}^{z}
=&\
\inf_{\widehat{\mathcal{NN}}_{z,i,1+d,n\times d}^{\varrho,L,\mathbf{n}}}
\mathbb{E}\left[
\left|
\widetilde{\widehat{\mathcal{V}}}_i^m
-
\widehat{\mathcal{NN}}_{z,i,1+d,n\times d}^{\varrho,L,\mathbf{n}}
(p_i^m,\theta_{z,i}^{r_j})
\right|^2
\right].
\end{aligned}
\label{Deep-NN-E--YZK-2R}
\end{equation}

The result is obtained under one of the following additional assumptions.

\textbf{Assumption 1.}
The terminal function \(g\) is Lipschitz and satisfies \(g\geq\Phi\). Moreover,
\(\Phi\) is \(C^1\), and \(\Phi,D_x\Phi\) are Lipschitz.

\textbf{Assumption 2.}
The terminal function \(g\) is Lipschitz and satisfies \(g\geq\Phi\). Moreover,
\(\sigma\) is \(C^1\), with \(\sigma,D_x\sigma\) both Lipschitz, and \(\Phi\) is
\(C^2\), with \(\Phi,D_x\Phi,D_x^2\Phi\) all Lipschitz.

For later use, we separate the time-discretization errors of the reflected BSDE into
the \(Y\)-component and the \(Z\)-component:
\begin{equation}
\begin{aligned}
\varepsilon_Y(h)
:=
\max_{0\leq i\leq N-1}
\mathbb{E}\left[
\left|
Y_{t_i}-Y_i^{\pi,m}
\right|^2
\right],
\qquad
\varepsilon_Z(h)
:=
\mathbb{E}\left[
\sum_{i=0}^{N-1}
\int_{t_i}^{t_{i+1}}
\left|
Z_t-Z_i^{\pi,m}
\right|^2dt
\right].
\end{aligned}
\label{RDBR-eps-YZ}
\end{equation}
According to \cite{BBGJF08}, we have
\[
\varepsilon_Y(h)=O(h^{1/2})
\quad\text{under Assumption 1},
\qquad
\varepsilon_Y(h)=O(h)
\quad\text{under Assumption 2},
\]
while
\[
\varepsilon_Z(h)=O(h^{1/2}).
\]

\begin{theorem}
\label{ae-DMCE-R}
Let the assumptions (i)--(iv) hold together with either Assumption 1 or
Assumption 2. Let \((Y_{t_i},Z_{t_i})\) and
\((\widehat{\mathcal{Y}}_i,\widehat{\mathcal{Z}}_i)\) be the solution of the
RFBSDE \eqref{eq:RBSDE} and the solution of the RDBR method \eqref{NS-DMC-R},
respectively. Then, for small enough \(h\), we have
\begin{equation}
\begin{aligned}
&\max_{0\leq i\leq N-1}
\mathbb{E}\left[
\left|
Y_{t_i}-\widehat{\mathcal{Y}}_i
\right|^2
\right]
+
\mathbb{E}\left[
\sum_{i=0}^{N-1}
\int_{t_i}^{t_{i+1}}
\left|
Z_s-\widehat{\mathcal{Z}}_i
\right|^2ds
\right]\\
&\leq
C\varepsilon_Y(h)
+
C\varepsilon_Z(h)
+
C\sum_{i=0}^{N-1}
\left(
N\widehat{\varepsilon}_{i}^{y}
+
\widehat{\varepsilon}_{i}^{z}
+
\frac{N}{K_i}
\right).
\end{aligned}
\label{RDBR-error-estimate}
\end{equation}
Consequently, using \eqref{RDBR-eps-YZ}, the time-discretization contribution is
of order \(O(h^{1/2})\) under Assumption 1. Under Assumption 2, the \(Y\)-component
has order \(O(h)\), while the \(Z\)-component remains \(O(h^{1/2})\); hence the
combined estimate in \eqref{RDBR-error-estimate} is still of order \(O(h^{1/2})\)
unless a sharper estimate for \(\varepsilon_Z(h)\) is available.
\end{theorem}

\begin{proof}
From \cite{BBGJF08}, we have
\begin{equation}
\left\{
\begin{aligned}
\max_{0\leq i\leq N-1}
\mathbb{E}\left[
\left|
Y_{t_i}-Y_i^{\pi,m}
\right|^2
\right]
&=
\varepsilon_Y(h),\\
\mathbb{E}\left[
\sum_{i=0}^{N-1}
\int_{t_i}^{t_{i+1}}
\left|
Z_t-Z_i^{\pi,m}
\right|^2dt
\right]
&=
\varepsilon_Z(h).
\end{aligned}
\right.
\label{RDBR-disc-error}
\end{equation}
Moreover,
\[
\varepsilon_Y(h)=O(h^{1/2})
\quad\text{under Assumption 1},
\qquad
\varepsilon_Y(h)=O(h)
\quad\text{under Assumption 2},
\]
and
\[
\varepsilon_Z(h)=O(h^{1/2}).
\]

From \eqref{NS-M-R} and \eqref{Deep-NN-E-R}, we have, for
\(i\in\{0,1,\ldots,N-1\}\),
\begin{equation}
\begin{aligned}
\overline{Y}_i^{\pi,m}-\mathcal{U}_i^m
=&\
\mathbb{E}_i\left[
Y_{i+1}^{\pi,m}
-
\widehat{\mathcal{Y}}_{i+1}^{m}
\right]\\
&+
h\left(
f(t_i,X_i^{\pi,m},\overline{Y}_i^{\pi,m},Z_i^{\pi,m})
-
f(t_i,X_i^{\pi,m},\mathcal{U}_i^m,\widehat{\mathcal{V}}_i^m)
\right).
\end{aligned}
\label{RDBR-Y-local-identity}
\end{equation}
Proceeding similarly as Step 1 in the proof of Theorem \ref{ae-DMCE}, by Young
inequality, Cauchy-Schwarz inequality, and the Lipschitz condition with respect to
\(f\), we obtain
\begin{equation}
\begin{aligned}
\mathbb{E}\left[
\left|
\overline{Y}_i^{\pi,m}-\mathcal{U}_i^m
\right|^2
\right]
&\leq
(1+\gamma h)
\mathbb{E}\left[
\left|
\mathbb{E}_i\left[
Y_{i+1}^{\pi,m}
-
\widehat{\mathcal{Y}}_{i+1}^{m}
\right]
\right|^2
\right]\\
&\quad
+2L^2h^2\left(1+\frac{1}{\gamma h}\right)
\mathbb{E}\left[
\left|
\overline{Y}_i^{\pi,m}-\mathcal{U}_i^m
\right|^2
+
\left|
Z_i^{\pi,m}-\widehat{\mathcal{V}}_i^m
\right|^2
\right].
\end{aligned}
\label{RDBR-Y-local-1}
\end{equation}
From Cauchy-Schwarz inequality and the tower property of conditional expectations,
one yields
\begin{equation}
\begin{aligned}
\mathbb{E}\left[
\left|
Z_i^{\pi,m}-\widehat{\mathcal{V}}_i^m
\right|^2
\right]
&\leq
\frac{2d}{h}
\mathbb{E}\left[
\left|
Y_{i+1}^{\pi,m}
-
\widehat{\mathcal{Y}}_{i+1}^{m}
\right|^2
\right]\\
&\quad
-\frac{2d}{h}
\mathbb{E}\left[
\left|
\mathbb{E}_i\left[
Y_{i+1}^{\pi,m}
-
\widehat{\mathcal{Y}}_{i+1}^{m}
\right]
\right|^2
\right].
\end{aligned}
\label{RDBR-Z-aux-1}
\end{equation}
Plugging \eqref{RDBR-Z-aux-1} into \eqref{RDBR-Y-local-1} and choosing
\(\gamma=4dL^2\), we obtain, for \(h\) small enough,
\begin{equation}
\begin{aligned}
\mathbb{E}\left[
\left|
\overline{Y}_i^{\pi,m}-\mathcal{U}_i^m
\right|^2
\right]
\leq
(1+Ch)
\mathbb{E}\left[
\left|
Y_{i+1}^{\pi,m}
-
\widehat{\mathcal{Y}}_{i+1}^{m}
\right|^2
\right].
\end{aligned}
\label{RDBR-Y-local-2}
\end{equation}

Next, by Young inequality as in Step 1 of the proof of Theorem \ref{ae-DMCE}, we have
\begin{equation}
\begin{aligned}
\mathbb{E}\left[
\left|
\overline{Y}_i^{\pi,m}-\overline{\mathcal{Y}}_i
\right|^2
\right]
&\leq
(1+Ch)
\mathbb{E}\left[
\left|
Y_{i+1}^{\pi,m}
-
\widehat{\mathcal{Y}}_{i+1}^{m}
\right|^2
\right]\\
&\quad
+
CN
\mathbb{E}\left[
\left|
\overline{\mathcal{Y}}_i-\mathcal{U}_i^m
\right|^2
\right].
\end{aligned}
\label{RDBR-Y-raw}
\end{equation}
By the same arguments as in Step 2 in the proof of Theorem \ref{ae-DMCE}, it follows
that, for \(h\) small enough,
\begin{equation}
\begin{aligned}
\mathbb{E}\left[
\left|
\overline{\mathcal{Y}}_i-\mathcal{U}_i^m
\right|^2
\right]
\leq
C\left(
\widehat{\varepsilon}_{i}^{y}
+
h\widehat{\varepsilon}_{i}^{z}
+
\frac{1}{K_i}
\right).
\end{aligned}
\label{RDBR-local-network-Y}
\end{equation}
Since
\[
Y_i^{\pi,m}
=
\max\left\{
\overline{Y}_i^{\pi,m},
\Phi(X_i^{\pi,m})
\right\},
\qquad
\widehat{\mathcal{Y}}_i
=
\max\left\{
\overline{\mathcal{Y}}_i,
\Phi(X_i^{\pi,m})
\right\},
\]
and
\[
|\max(a,c)-\max(b,c)|\leq |a-b|,
\]
we obtain from \eqref{RDBR-Y-raw} and \eqref{RDBR-local-network-Y} that
\begin{equation}
\begin{aligned}
\mathbb{E}\left[
\left|
Y_i^{\pi,m}-\widehat{\mathcal{Y}}_i
\right|^2
\right]
&\leq
(1+Ch)
\mathbb{E}\left[
\left|
Y_{i+1}^{\pi,m}
-
\widehat{\mathcal{Y}}_{i+1}^{m}
\right|^2
\right]\\
&\quad
+
C\left(
\widehat{\varepsilon}_{i}^{y}
+
h\widehat{\varepsilon}_{i}^{z}
+
\frac{1}{K_i}
\right).
\end{aligned}
\label{RDBR-discrete-recursion}
\end{equation}
Using the terminal identity
\[
Y_N^{\pi,m}=g(X_N^{\pi,m})=\widehat{\mathcal{Y}}_N,
\]
and applying the discrete Gronwall inequality backward in time, we derive
\begin{equation}
\begin{aligned}
\max_{0\leq i\leq N-1}
\mathbb{E}\left[
\left|
Y_i^{\pi,m}-\widehat{\mathcal{Y}}_i
\right|^2
\right]
\leq
C\sum_{i=0}^{N-1}
\left(
\widehat{\varepsilon}_{i}^{y}
+
h\widehat{\varepsilon}_{i}^{z}
+
\frac{1}{K_i}
\right).
\end{aligned}
\label{RDBR-Y-discrete-error}
\end{equation}
Combining \eqref{RDBR-Y-discrete-error} with \eqref{RDBR-disc-error}, we deduce
\begin{equation}
\begin{aligned}
\max_{0\leq i\leq N-1}
\mathbb{E}\left[
\left|
Y_{t_i}-\widehat{\mathcal{Y}}_i
\right|^2
\right]
\leq
C\varepsilon_Y(h)
+
C\sum_{i=0}^{N-1}
\left(
\widehat{\varepsilon}_{i}^{y}
+
h\widehat{\varepsilon}_{i}^{z}
+
\frac{1}{K_i}
\right).
\end{aligned}
\label{RDBR-Y-error}
\end{equation}
This proves the estimate for the \(Y\)-component.

It remains to estimate the \(Z\)-component. By Young inequality and the definition
of \(\varepsilon_Z(h)\), we have
\begin{equation}
\begin{aligned}
&\mathbb{E}\left[
\sum_{i=0}^{N-1}
\int_{t_i}^{t_{i+1}}
\left|
Z_s-\widehat{\mathcal{Z}}_i
\right|^2ds
\right]\\
&\leq
2\mathbb{E}\left[
\sum_{i=0}^{N-1}
\int_{t_i}^{t_{i+1}}
\left|
Z_s-Z_i^{\pi,m}
\right|^2ds
\right]
+
2\sum_{i=0}^{N-1}
h\mathbb{E}\left[
\left|
Z_i^{\pi,m}-\widehat{\mathcal{Z}}_i
\right|^2
\right]\\
&\leq
2\varepsilon_Z(h)
+
2\sum_{i=0}^{N-1}
h\mathbb{E}\left[
\left|
Z_i^{\pi,m}-\widehat{\mathcal{Z}}_i
\right|^2
\right].
\end{aligned}
\label{RDBR-Z-split}
\end{equation}
For each \(i=0,\ldots,N-1\), since
\[
Z_i^{\pi,m}
=
\mathbb{E}_i\left[
Y_{i+1}^{\pi,m}\frac{(\Delta W_i^m)^\top}{h}
\right],
\qquad
\widehat{\mathcal{V}}_i^m
=
\mathbb{E}_i\left[
\widehat{\mathcal{Y}}_{i+1}^{m}
\frac{(\Delta W_i^m)^\top}{h}
\right],
\]
Cauchy-Schwarz inequality yields
\begin{equation}
\begin{aligned}
h\mathbb{E}\left[
\left|
Z_i^{\pi,m}-\widehat{\mathcal{V}}_i^m
\right|^2
\right]
\leq
C\mathbb{E}\left[
\left|
Y_{i+1}^{\pi,m}
-
\widehat{\mathcal{Y}}_{i+1}^{m}
\right|^2
\right].
\end{aligned}
\label{RDBR-Z-target-error}
\end{equation}
Moreover, by the same regression argument as in Step 2 of the proof of Theorem
\ref{ae-DMCE}, together with the Monte Carlo estimate for the \(Z\)-label, we have
\begin{equation}
\begin{aligned}
h\mathbb{E}\left[
\left|
\widehat{\mathcal{V}}_i^m-\widehat{\mathcal{Z}}_i
\right|^2
\right]
\leq
C\left(
h\widehat{\varepsilon}_{i}^{z}
+
\frac{1}{K_i}
\right).
\end{aligned}
\label{RDBR-Z-network-error}
\end{equation}
Combining \eqref{RDBR-Z-target-error} and \eqref{RDBR-Z-network-error}, we obtain
\begin{equation}
\begin{aligned}
h\mathbb{E}\left[
\left|
Z_i^{\pi,m}-\widehat{\mathcal{Z}}_i
\right|^2
\right]
\leq
C\mathbb{E}\left[
\left|
Y_{i+1}^{\pi,m}
-
\widehat{\mathcal{Y}}_{i+1}^{m}
\right|^2
\right]
+
C\left(
h\widehat{\varepsilon}_{i}^{z}
+
\frac{1}{K_i}
\right).
\end{aligned}
\label{RDBR-Z-local}
\end{equation}
Summing \eqref{RDBR-Z-local} over \(i=0,\ldots,N-1\), and using
\eqref{RDBR-Y-discrete-error}, we get
\begin{equation}
\begin{aligned}
\sum_{i=0}^{N-1}
h\mathbb{E}\left[
\left|
Z_i^{\pi,m}-\widehat{\mathcal{Z}}_i
\right|^2
\right]
&\leq
C\sum_{i=0}^{N-1}
\mathbb{E}\left[
\left|
Y_{i+1}^{\pi,m}
-
\widehat{\mathcal{Y}}_{i+1}^{m}
\right|^2
\right]
+
C\sum_{i=0}^{N-1}
\left(
h\widehat{\varepsilon}_{i}^{z}
+
\frac{1}{K_i}
\right)\\
&\leq
C\sum_{i=0}^{N-1}
\left(
N\widehat{\varepsilon}_{i}^{y}
+
\widehat{\varepsilon}_{i}^{z}
+
\frac{N}{K_i}
\right).
\end{aligned}
\label{RDBR-Z-discrete-network}
\end{equation}
Inserting \eqref{RDBR-Z-discrete-network} into \eqref{RDBR-Z-split}, we obtain
\begin{equation}
\begin{aligned}
\mathbb{E}\left[
\sum_{i=0}^{N-1}
\int_{t_i}^{t_{i+1}}
\left|
Z_s-\widehat{\mathcal{Z}}_i
\right|^2ds
\right]
\leq
C\varepsilon_Z(h)
+
C\sum_{i=0}^{N-1}
\left(
N\widehat{\varepsilon}_{i}^{y}
+
\widehat{\varepsilon}_{i}^{z}
+
\frac{N}{K_i}
\right).
\end{aligned}
\label{RDBR-Z-error}
\end{equation}
Finally, adding \eqref{RDBR-Y-error} and \eqref{RDBR-Z-error}, and using
\[
\sum_{i=0}^{N-1}
\left(
\widehat{\varepsilon}_{i}^{y}
+
h\widehat{\varepsilon}_{i}^{z}
+
\frac{1}{K_i}
\right)
\leq
C\sum_{i=0}^{N-1}
\left(
N\widehat{\varepsilon}_{i}^{y}
+
\widehat{\varepsilon}_{i}^{z}
+
\frac{N}{K_i}
\right),
\]
we obtain \eqref{RDBR-error-estimate}. This completes the proof.
\end{proof}

\section{Numerical experiments}

In this section, we demonstrate the performance of \textbf{Algorithm 1}
to solve the high dimensional nonlinear PDEs (\ref{PDE}) via two numerical examples.

All the numerical tests are implemented in Python 3.8 on a desktop computer with  Intel(R) Xeon(R) Gold 6133 CPU (2.50GHz)
and 64 GB RAM (3200MHz), and an NVIDIA GeForce RTX 4090  GPU (24 GB) with CUDA 11.8 support. All implementations were developed using PyTorch 2.0.
The corresponding code is available at: https://github.com/22w2e/A-DNMC-algorithm-for-high-dimensional-nonlinear-PDEs
and the results are recorded. Note that we still implement the proposed deep method by the feedforward neural network
because the computational cost of the GroupSort network is bigger than that of the costly feedforward neural network
in practice (see \cite{GMPHWX22}).
To mitigate randomness inherent in
 stochastic sampling, each experiment is independently repeated $10$ times under identical conditions,
 and the corresponding results are reported.
Although a larger sample size would tighten the confidence intervals, the computational cost for these high-dimensional problems is significant. Ten repetitions allow us to obtain a reliable estimate of the mean error magnitude and variance, which is consistent with standard practices in related literature \cite{HCPHWX20}. 

For notational clarity in the numerical tables, \(u(0,x_0)\) denotes the exact
analytical value at the initial point. For the \(\ell\)-th independent run, we denote
the corresponding numerical estimate by \(u_{\ell}^{\pi}(0,x_0)\). The sample mean
reported in the tables is denoted by
\[
\overline{u}^{\,\pi}(0,x_0)
:=
\frac{1}{N_{\mathrm{rep}}}
\sum_{\ell=1}^{N_{\mathrm{rep}}}
u_{\ell}^{\pi}(0,x_0),
\qquad
N_{\mathrm{rep}}=10.
\]
The reported mean absolute error and relative error are computed by
\[
\mathbb{E}\left[
\left|
u(0,x_0)-u^{\pi}(0,x_0)
\right|
\right]
\approx
\frac{1}{N_{\mathrm{rep}}}
\sum_{\ell=1}^{N_{\mathrm{rep}}}
\left|
u(0,x_0)-u_{\ell}^{\pi}(0,x_0)
\right|,
\]
and
\[
\mathrm{Rel.\ err.}
=
\frac{
\left|
u(0,x_0)-\overline{u}^{\,\pi}(0,x_0)
\right|
}{
\left|
u(0,x_0)
\right|
},
\]
respectively.

\textbf{Example 1.}
We first consider the following $d$-dimensional nonlinear parabolic PDE (taken from \cite{EWHJJA17}):
\begin{equation}
\label{Example-1}
\begin{cases}
\displaystyle
\frac{\partial u}{\partial t}(t,x)
+  \frac{\sigma^2}{2}\sum_{\ell=1}^{d}D^2_{x_\ell}u(t,x)
+ \Big(u(t,x) - \frac{d+2}{2d}\Big)
    \Big(\sigma \sum_{\ell=1}^{d}  D_{x_\ell}u(t,x)\Big)
= 0,\\ 
\displaystyle
u(T,x)=\frac{\exp\!\left(T+\tfrac{1}{d}\sum_{\ell=1}^{d}x_\ell\right)}
           {1+\exp\!\left(T+\tfrac{1}{d}\sum_{\ell=1}^{d}x_\ell\right)} .
\end{cases}
\end{equation}

The corresponding analytic solution is explicitly given by
\[
u(t,x)
=\frac{\exp\!\left(t+\tfrac{1}{d}\sum_{\ell=1}^{d}x_\ell\right)}
        {1+\exp\!\left(t+\tfrac{1}{d}\sum_{\ell=1}^{d}x_\ell\right)}.
\]

In the numerical experiments we take the terminal time $T=1$,
the number of time steps $N=10$, the Monte Carlo samples $M=10000$,
and the initial state $x_0=(0,0,\dots,0)^\top \in \mathbb{R}^d$. Thus, the exact solution of \eqref{Example-1} at the initial point is
\(u(0,x_0)=0.5\). The numerical estimate obtained in each independent run is denoted
by \(u_{\ell}^{\pi}(0,x_0)\), and its sample mean is denoted by
\(\overline{u}^{\,\pi}(0,x_0)\).
The numerical solver is implemented using a deep learning framework where two sub-networks are trained at each time step to approximate the solution $u$ and its gradient $\nabla u$. Each network consists of two hidden layers with $d+110$ units and $\tanh$ activations. We utilize the Adam optimizer with a learning rate of $5 \times 10^{-4}$ and conduct 6,000 training iterations per step to ensure the convergence of the loss function.

\begin{figure}[H]
\centering
\subfigure[\footnotesize  $u(t, x)$ and its estimate at time $t=0.2$.]
{\includegraphics[width=6cm,height=3.1cm]{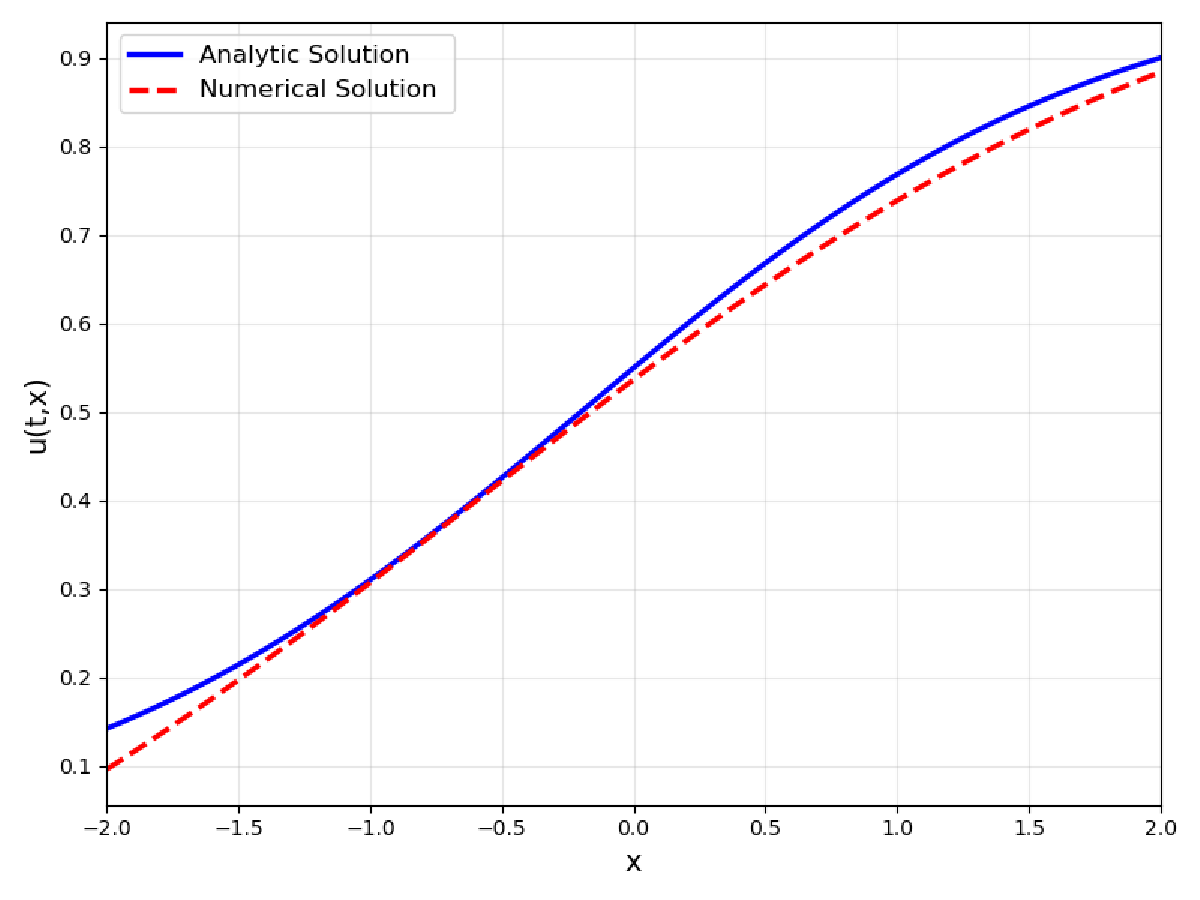}}~~~~
\subfigure[\footnotesize  $u(t, x)$ and its estimate at time $t=0.4$.]
{\includegraphics[width=6cm,height=3.1cm]{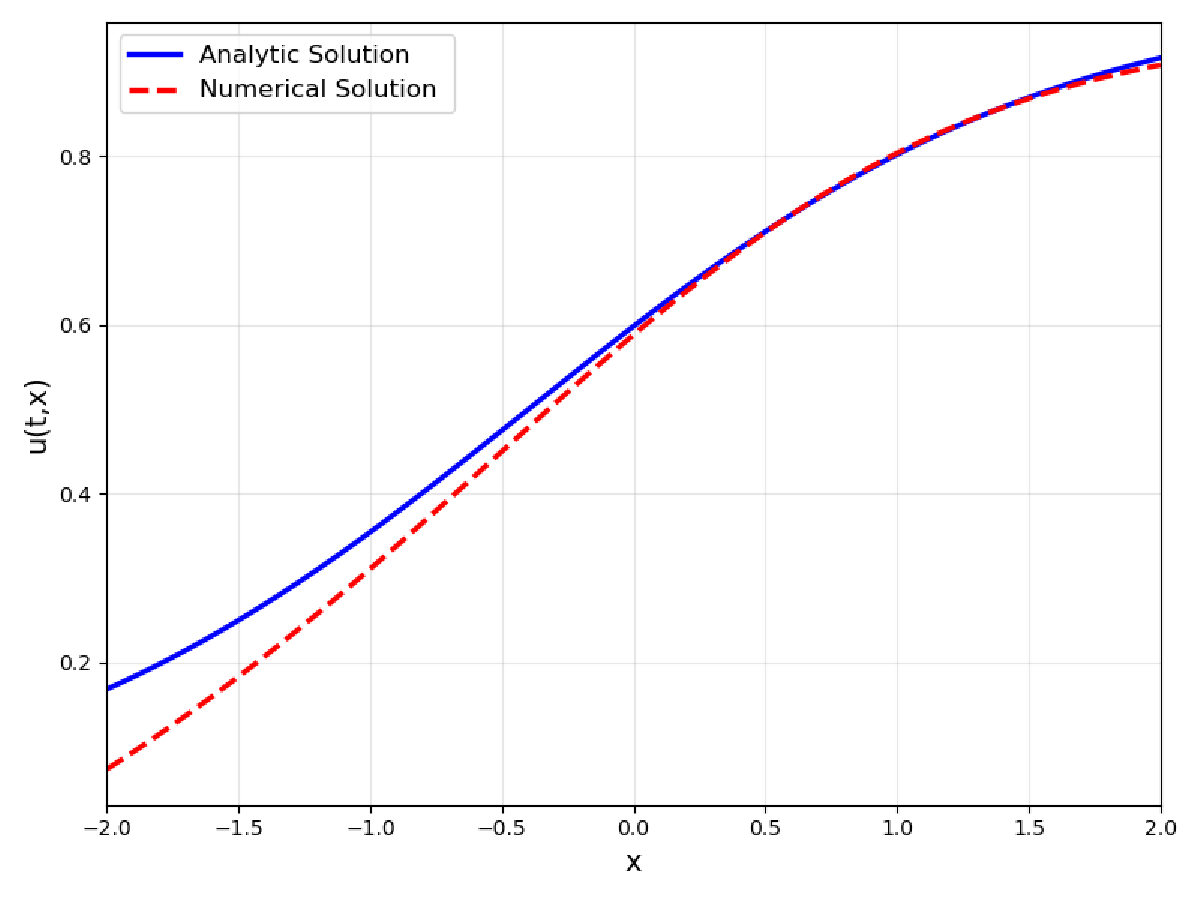}}
\\
\subfigure[\footnotesize  $u(t, x)$ and its estimate at time $t=0.6$.]
{\includegraphics[width=6cm,height=3.1cm]{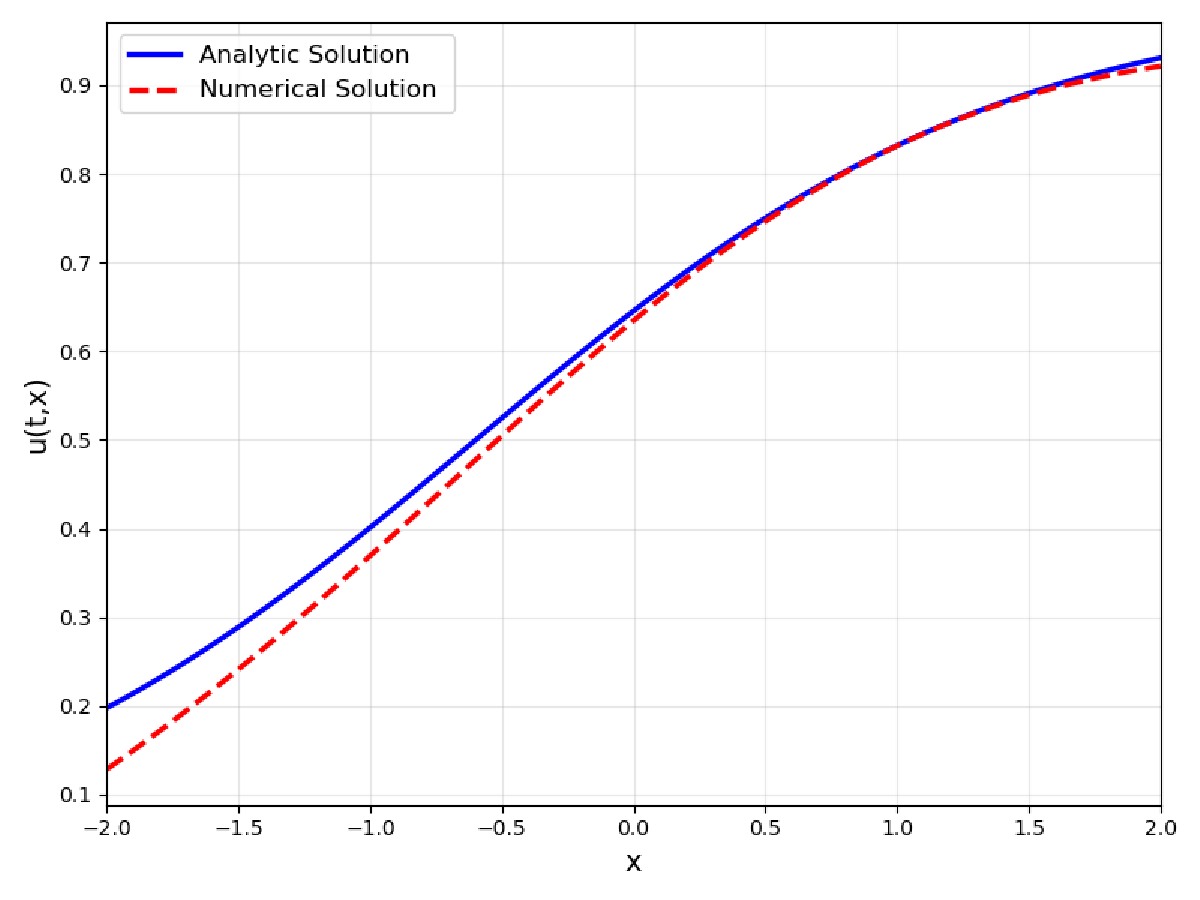}}~~~~
\subfigure[\footnotesize  $u(t, x)$ and its estimate at time $t=0.8$.]
{\includegraphics[width=6cm,height=3.1cm]{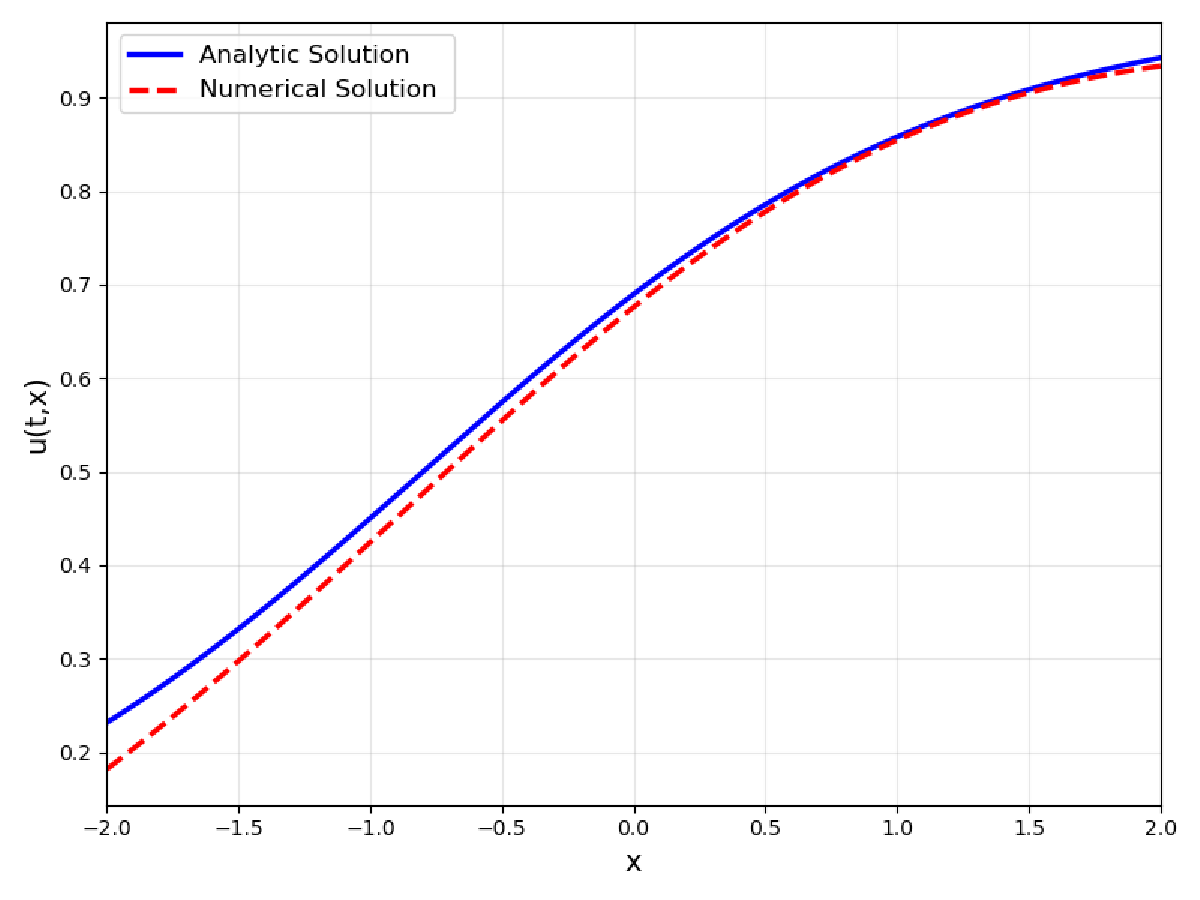}}
\caption{\footnotesize Estimated solution  $\overline{\mathfrak{u}}(t,x)$ obtained by DBR versus exact solution $u(t,x)$ for Example~1 with $d=1$.}\label{fex1}
\end{figure}

Figure~\ref{fex1} provides a qualitative visualization of the DBR algorithm's approximation accuracy across the entire time domain. The figure displays the estimated solution $\overline{\mathfrak{u}}(t, x)$ alongside the exact analytical solution $u(t, x)$ at four distinct time snapshots ($t=0.2, 0.4, 0.6, 0.8$) for the one-dimensional case. Across all four subplots, the estimated values exhibit near-perfect alignment with the true solution, with no visible divergence or significant error at any intermediate time step. This visual evidence demonstrates that the high accuracy achieved by the DBR method is not confined to the initial condition at $t=0$, but is consistently maintained throughout the entire time horizon from $t=0$ to $t=T$. The results confirm the DBR scheme's capability to correctly propagate the solution backward in time through its learned network functions, validating the effectiveness of the expectation-based loss functions in capturing the underlying dynamics of the PDE.

\begin{table}[!htb]
  \centering
  \begin{threeparttable}
    \caption{Comparison of numerical results for different schemes in Example~1 across different dimensions.}
    \label{tex1}
    \setlength{\tabcolsep}{1.3mm}{
   \begin{tabular}{lcccccc}
\toprule
Scheme
& \(d\)
& \(u(0,x_0)\)
& \(\overline{u}^{\,\pi}(0,x_0)\)
& Std. dev.
& \(\mathbb{E}[|u-u^{\pi}|]\)
& Rel. err.\\
\midrule
      \multirow{4}{*}{DBR}             & 50  & 0.500000 & 0.517697 & 0.007637 & 0.017697 & 3.5393\% \\
                                       & 80  & 0.500000 & 0.508568 & 0.007589 & 0.008683 & 1.7137\% \\
                                       & 100 & 0.500000 & 0.507220 & 0.005692 & 0.008006 & 1.4441\%  \\
                                       &200 & 0.500000 & 0.482896 & 0.007052 & 0.017104 & 3.4207\% \\
      \midrule
      \multirow{4}{*}{
    DBDP1~\cite{HCPHWX20}}             & 50  & 0.500000 & 0.569523 & 0.008724 & 0.069523 & 13.9047\%   \\
                                       & 80  & 0.500000 & 0.548494 & 0.005876 & 0.048494 & 9.6989\% \\
                                       & 100 & 0.500000 & 0.545578 & 0.005622 & 0.045578 & 9.1155\%  \\
                                       &200 & 0.500000 & 0.524425 & 0.002196 & 0.024425 & 4.8850\% \\
      \bottomrule
    \end{tabular}}
  \end{threeparttable}
\end{table}

Table~\ref{tex1} complements these visual findings by providing a quantitative comparison between the proposed DBR algorithm and the existing DBDP1 method in \cite{HCPHWX20} across multiple high dimensions ($50, 80, 100$ and $200$).
For each scheme and dimension, we report the exact value \(u(0,x_0)\), the sample
mean \(\overline{u}^{\,\pi}(0,x_0)\) over 10 independent runs, the empirical standard
deviation \(\mathrm{Std.\ dev.}(u^{\pi}(0,x_0))\), the mean absolute error
\[
\mathbb{E}\left[
\left|
u(0,x_0)-u^{\pi}(0,x_0)
\right|
\right],
\]
and the relative error
\[
\mathrm{Rel.\ err.}
=
\frac{
\left|
u(0,x_0)-\overline{u}^{\,\pi}(0,x_0)
\right|
}{
\left|
u(0,x_0)
\right|
}.
\]
Regarding accuracy, the DBR algorithm maintains a relative error consistently below $4\%$ across all tested dimensions, significantly outperforming the DBDP1 method which exhibits errors as high as $13.9\%$ at dimension $50$ and $9.7\%$ at dimension $80$.  More importantly, the mean absolute error, which represents the average magnitude of the absolute deviation between the true solution and the numerical approximation, is consistently smaller for the DBR method across all dimensions. For instance, at dimension $100$, DBR achieves a mean absolute error of $0.008006$ compared to DBDP1's $0.045578$, indicating that the DBR algorithm produces approximations that are substantially closer to the true solution on average. This observation is consistent with the motivation of the DBR scheme, namely that
conditional-expectation-based labels can reduce pathwise noise in the training
targets.

\textbf{Example 2.}
(see \cite{GMPHWX22})
Consider the following high-dimensional nonlinear PDE with an unbounded solution
and a more complex structure:
\begin{equation}
\begin{aligned}
\begin{cases}
\partial_t u + \frac{1}{2} \operatorname{Tr}\left(\frac{I_d}{\sqrt{d}} (\frac{I_d}{\sqrt{d}})^{\top} D_x^2 u\right)+f\left(\cdot, \cdot, u, \sigma^{\top} D_x u\right)=0 & \text { on }[0, T) \times \mathbb{R}^d, \\
u(T, \cdot)=g,
\end{cases}
\end{aligned}
\label{PDE-2}
\end{equation}
where the generator
$$
\begin{aligned}
f(t, x, y, z)=\hat{k}(t, x)-\frac{y}{\sqrt{d}}\left(1_d \cdot z\right)-\frac{y^2}{2},
\end{aligned}
$$
with $-\hat{k}(t, x)=\partial_t u+\frac{1}{2 d} \operatorname{Tr}\left(D_x^2 u\right)+\frac{u}{\sqrt{d}} D_{x_i} u+\frac{u^2}{2}$.
Thus, the solution of PDE (\ref{PDE-2}) is given by
$u(t, x)=\frac{T-t}{d} \sum_{i=1}^d\left(\sin \left(x_i\right) 1_{x_i<0}+x_i 1_{x_i \geq 0}\right)+\cos \left(\sum_{i=1}^d i x_i\right).$

Example~2 presents a more challenging class of high-dimensional nonlinear PDEs characterized by an unbounded solution and a complex structural form. 
We note that Example 2 involves an unbounded solution. Hence this benchmark is used
to test the robustness of the proposed algorithm beyond the bounded classical setting.
The theoretical estimates above are stated under the regularity and integrability
conditions specified in Sections 2 and 3.

To comprehensively evaluate the performance of the proposed DBR algorithm, we conduct systematic numerical experiments across dimensions $d=1, 2, 8, 15, 20$ with varying time steps $N$, where the initial point is fixed as $x_0=0.5\mathbf{1}_d$ for all cases.
The neural network architecture consists of two hidden layers with hyperbolic tangent activation functions, where each hidden layer contains $d + 10$ units. The model is trained using a batch size of 400 and a total of 3000 iterations.

\begin{figure}[H]
\centering
\subfigure[\footnotesize  $u(t, x)$ and its estimate at time $t=0.3$.]
{\includegraphics[width=6cm,height=3.1cm]{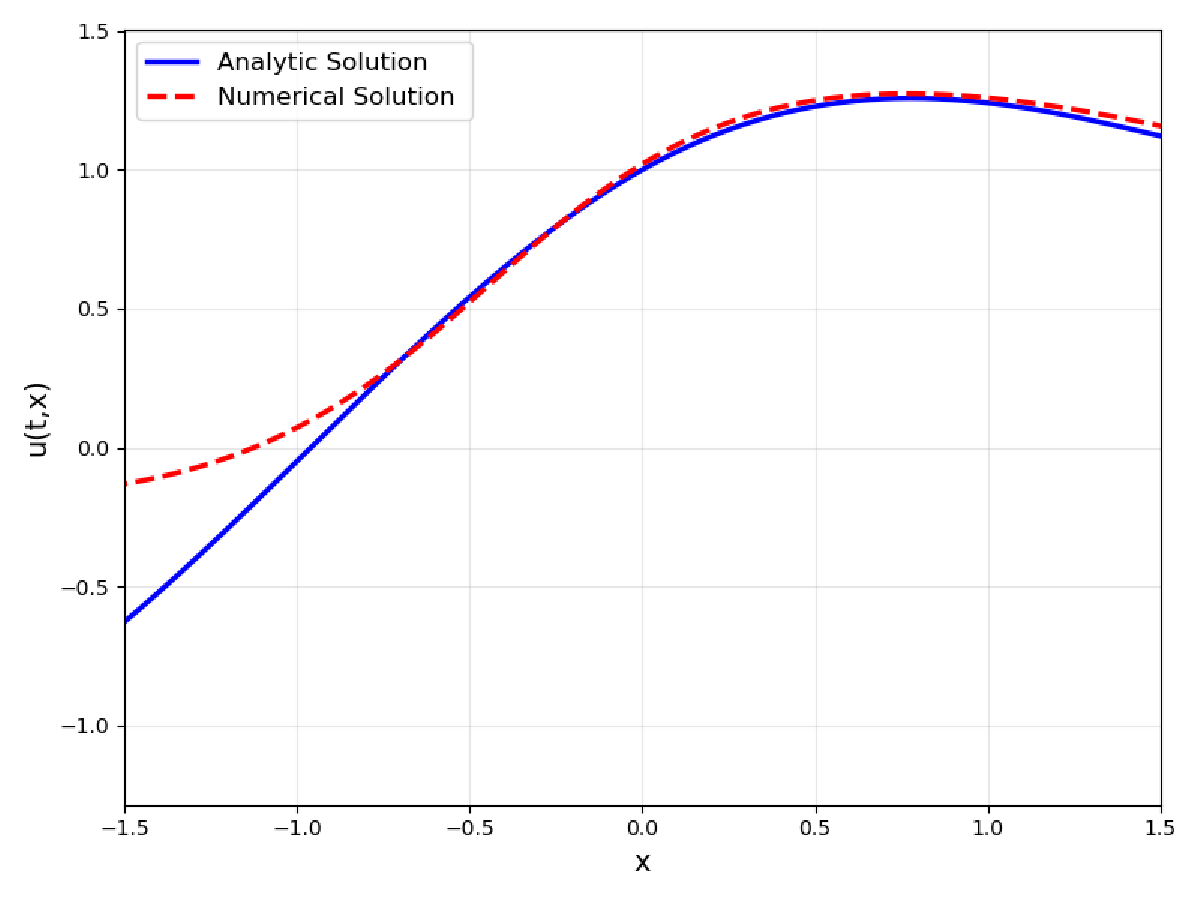}}~~~~
\subfigure[\footnotesize  $u(t, x)$ and its estimate at time $t=0.5$.]
{\includegraphics[width=6cm,height=3.1cm]{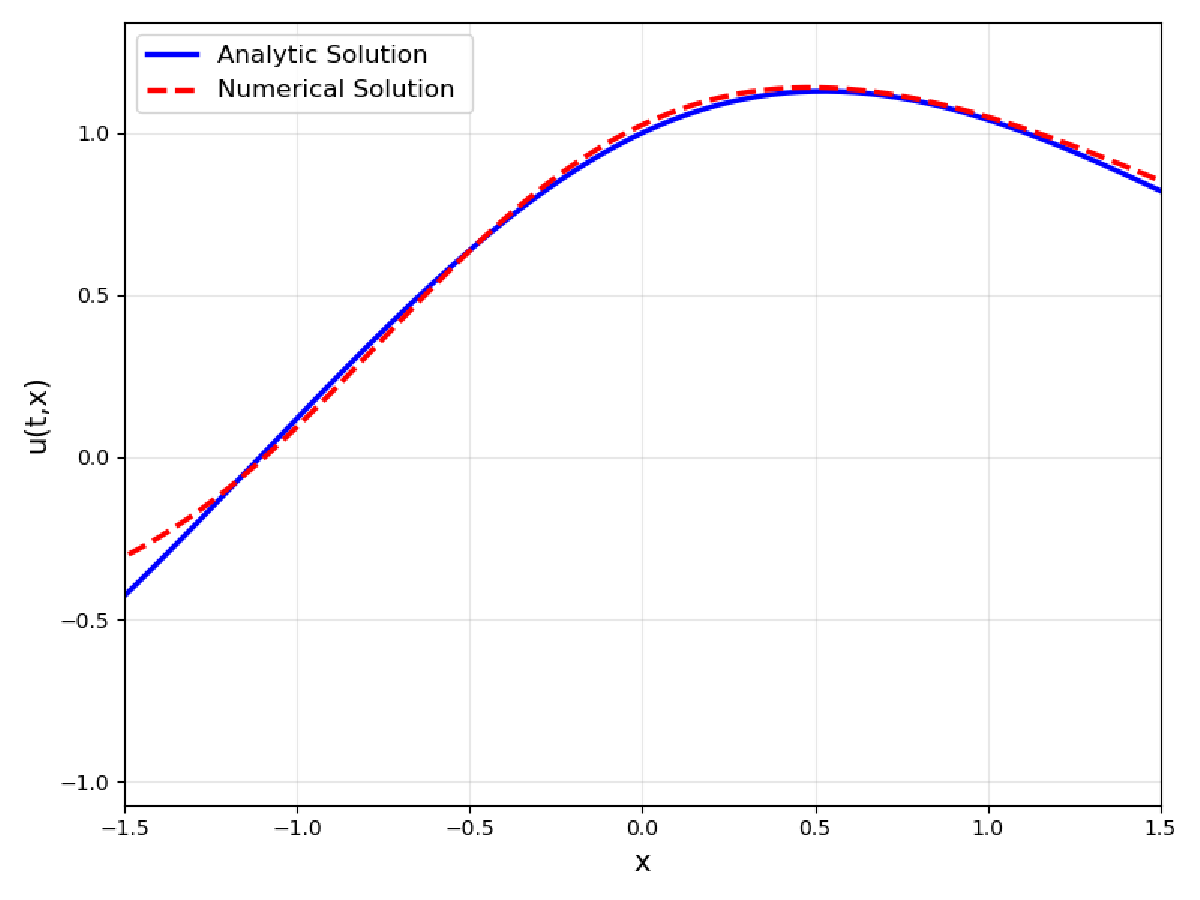}}
\\
\subfigure[\footnotesize  $u(t, x)$ and its estimate at time $t=0.7$.]
{\includegraphics[width=6cm,height=3.1cm]{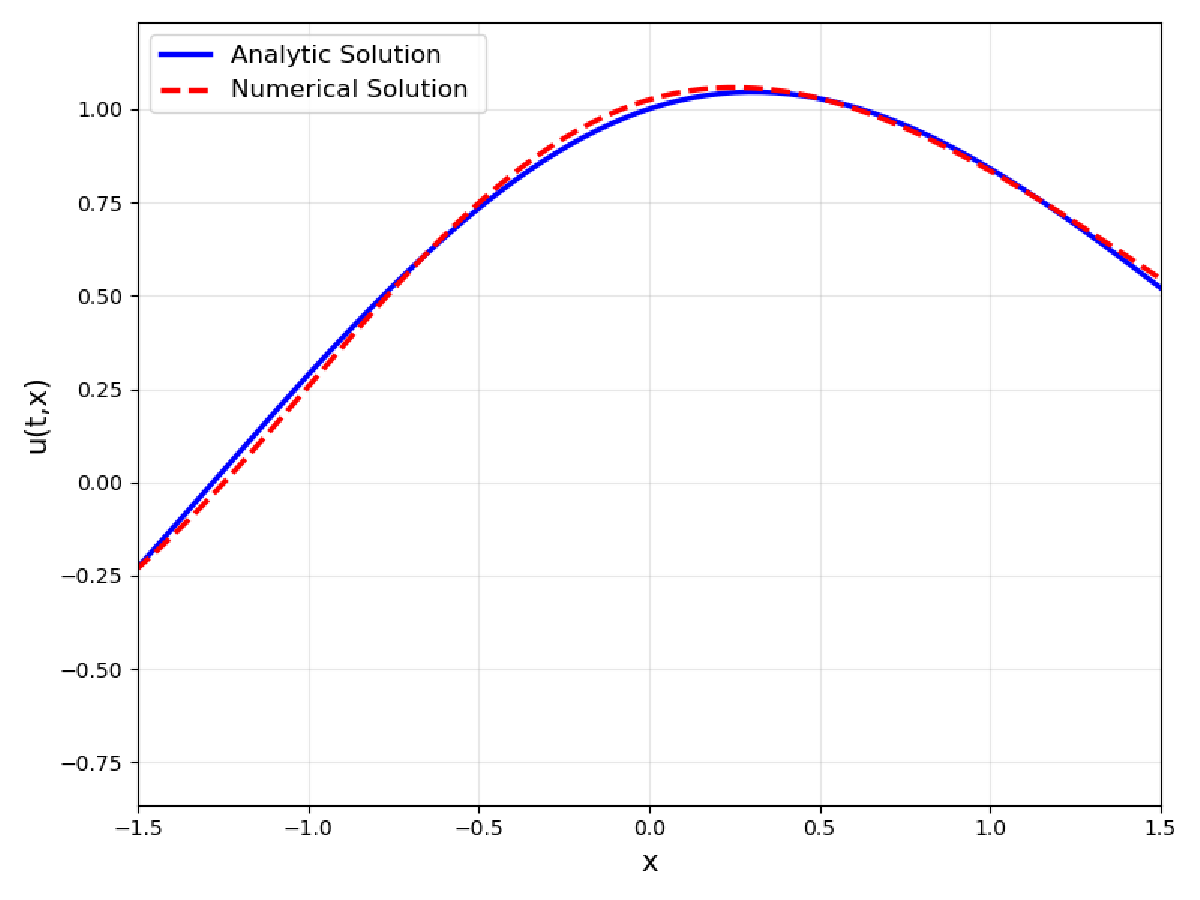}}~~~~
\subfigure[\footnotesize  $u(t, x)$ and its estimate at time $t=0.9$.]
{\includegraphics[width=6cm,height=3.1cm]{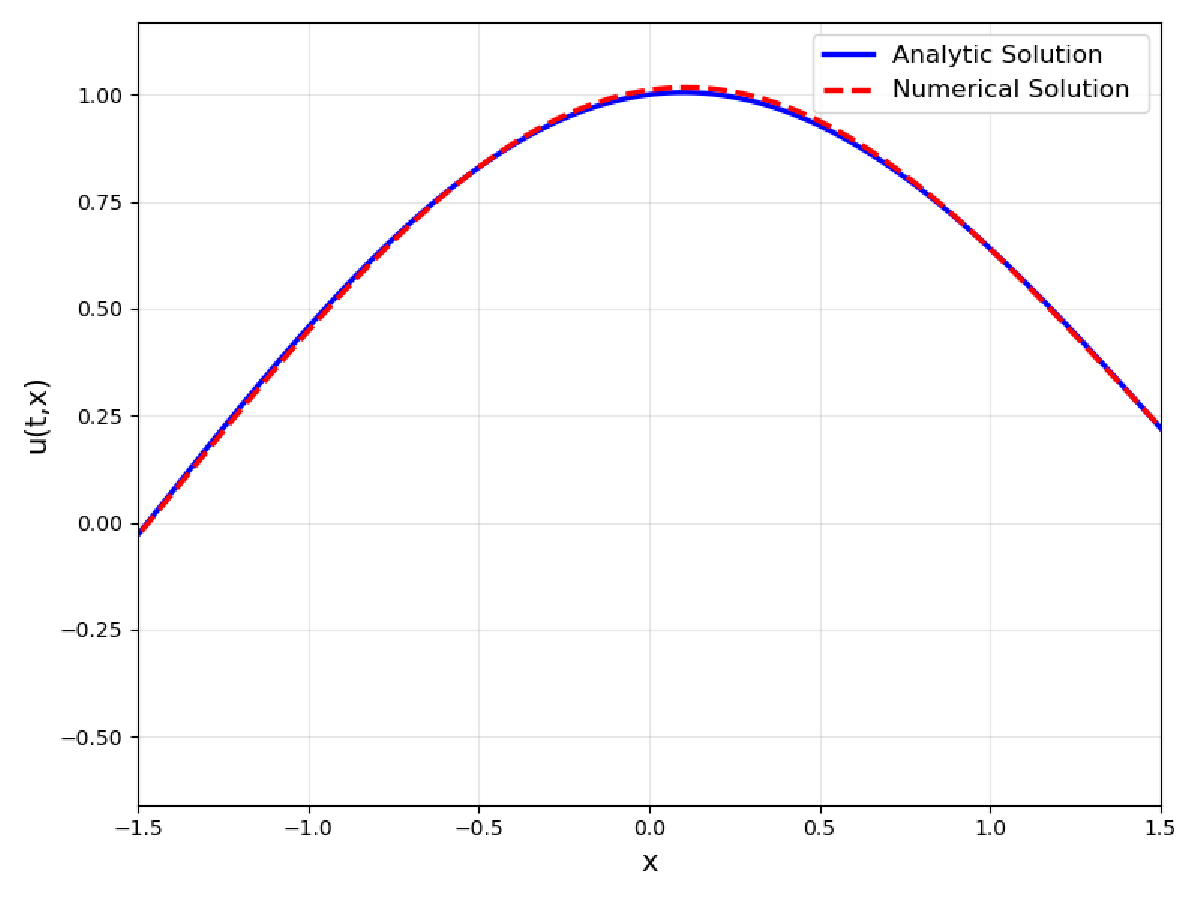}}
\caption{Estimated solution  $\overline{\mathfrak{u}}(t, x)$ obtained by DBR versus exact solution $u(t,x)$ for Example~2 with $d=1$.}\label{fex2}
\end{figure}

\begin{table}[!htb]
  \centering
  \begin{threeparttable}
    \caption{Comparison of numerical results for different schemes in Example~2 across different $N$ with $d=1$.}
    \label{tex2d1}
    \setlength{\tabcolsep}{1.3mm}{
   \begin{tabular}{lcccccc}
\toprule
Scheme
& \(N\)
& \(u(0,x_0)\)
& \(\overline{u}^{\,\pi}(0,x_0)\)
& Std. dev.
& \(\mathbb{E}[|u-u^{\pi}|]\)
& Rel. err.\\
\midrule
      \multirow{3}{*}{DBR} 
                                       & 10  & 1.377583 & 1.403209 & 0.011086 & 0.025700 & 1.8602\%  \\
                                       & 20  & 1.377583 & 1.400363 & 0.006224 & 0.022780 & 1.6536\%  \\
                                       & 30 & 1.377583 & 1.393180 & 0.006930 & 0.015598 & 1.1323\%  \\

      \midrule
      \multirow{3}{*}{
    DBDP1~\cite{HCPHWX20}} 
                                          & 10  & 1.377583 & 1,371497 & 0.023389 & 0.019112 & 0.4417\%   \\
                                         & 20  & 1.377583 & 1.366828 & 0.017810 & 0.017439 & 0.7806\%  \\
                                         & 30 & 1.377583 & 1.367642 & 0.019181 & 0.017814 & 0.7216\% \\

      \bottomrule
    \end{tabular}}
  \end{threeparttable}
\end{table}

\begin{table}[!htb]
  \centering
  \begin{threeparttable}
    \caption{Comparison of numerical results for different schemes in Example~2 across different $N$ with $d=2$.}
    \label{tex2d2}
    \setlength{\tabcolsep}{1.3mm}{
    \begin{tabular}{lcccccc}
\toprule
Scheme
& \(N\)
& \(u(0,x_0)\)
& \(\overline{u}^{\,\pi}(0,x_0)\)
& Std. dev.
& \(\mathbb{E}[|u-u^{\pi}|]\)
& Rel. err.\\
\midrule
      \multirow{3}{*}{DBR} 
                                       & 10  & 0.570737 & 0.571529 & 0.014380 & 0.012455 & 0.1387\%  \\
                                       & 20  & 0.570737 & 0.579191 & 0.016317 & 0.015263 & 1.4813\%  \\
                                       & 30 & 0.5707373 & 0.574875 & 0.016678 & 0.014892 & 0.7250\%  \\

      \midrule
      \multirow{3}{*}{
    DBDP1~\cite{HCPHWX20}} 
                                          & 10  & 0.570737 & 0.560685 & 0.009967 & 0.011556 & 1.7613\%   \\
                                         & 20  & 0.570737 & 0.563375 & 0.009665 & 0.009730 & 1.2899\%  \\
                                         & 30 & 0.570737 & 0.567637 & 0.007557 & 0.006504 & 0.5431\% \\

      \bottomrule
    \end{tabular}}
  \end{threeparttable}
\end{table}

\begin{table}[!htb]
  \centering
  \begin{threeparttable}
    \caption{Comparison of numerical results for different schemes in Example~2 across different $N$ with $d=8$.}
    \label{tex2d8}
    \setlength{\tabcolsep}{1.3mm}{
   \begin{tabular}{lcccccc}
\toprule
Scheme
& \(N\)
& \(u(0,x_0)\)
& \(\overline{u}^{\,\pi}(0,x_0)\)
& Std. dev.
& \(\mathbb{E}[|u-u^{\pi}|]\)
& Rel. err.\\
\midrule
      \multirow{3}{*}{DBR} 
                                       & 70  & 1.160317 & 1.240760 & 0.150684 & 0.131828 & 6.9328\%  \\
                                       & 90  & 1.160317 & 1.211207 & 0.116388 & 0.092109 & 4.3859\%  \\
                                       & 110 & 1.160317 & 1.232262 & 0.111441 & 0.113141 & 6.2005\%  \\

      \midrule
      \multirow{3}{*}{
    DBDP1~\cite{HCPHWX20}} 
                                          & 70  & 1.160317 & 1.126632 & 0.134559 & 0.107561 & 2.9030\%   \\
                                         & 90  & 1.160317 & 1.173978 & 0.077135 & 0.061737 & 1.1773\%  \\
                                         & 110 & 1.160317 & 1.121225 & 0.090274 & 0.075801 & 3.3690\% \\

      \bottomrule
    \end{tabular}}
  \end{threeparttable}
\end{table}

Figure \ref{fex2} provides a qualitative visualization of the approximation accuracy achieved by the DBR algorithm for the one-dimensional case of Example~2. The figure displays the estimated solution $\overline{\mathfrak{u}}(t, x)$ obtained by the DBR method alongside the exact analytical solution $u(t, x)$ at four distinct time snapshots (t=0.3, 0.5, 0.7, 0.9).
This visual evidence confirms that the DBR algorithm successfully captures the underlying dynamics of the PDE throughout the entire time horizon, demonstrating its capability to accurately propagate the solution backward in time through the learned network functions. The excellent agreement between the estimated and exact solutions serves as compelling qualitative validation of the effectiveness of the expectation-based loss functions employed in the DBR methodology.

For low-dimensional cases ($d=1, 2, 8$),  both the DBR and DBDP1 methods demonstrate satisfactory performance, achieving relative errors generally below $7\%$ across different $N$. As shown in Tables \ref{tex2d1}, \ref{tex2d2} and \ref{tex2d8}, the two algorithms exhibit comparable accuracy and stability, with no single method displaying a decisive advantage in these low-dimensional regimes. This observation aligns with the theoretical expectation that both probabilistic numerical schemes are well-suited for problems of low dimensionality.

\begin{table}[!htb]
  \centering
  \begin{threeparttable}
    \caption{Comparison of numerical results for different schemes in Example~2 across different $N$ with $d=15$.}
    \label{tex2d15}
    \setlength{\tabcolsep}{1.3mm}{
    \begin{tabular}{lcccccc}
\toprule
Scheme
& \(N\)
& \(u(0,x_0)\)
& \(\overline{u}^{\,\pi}(0,x_0)\)
& Std. dev.
& \(\mathbb{E}[|u-u^{\pi}|]\)
& Rel. err.\\
\midrule
      \multirow{3}{*}{DBR} 
                                       & 120  & -0.452413 & -0.491338 & 0.125350 & 0.112779 & 8.6040\%  \\
                                       & 130  & -0.452413 & -0.496170 & 0.194659 & 0.164798 & 9.6719\%  \\
                                       & 140 & -0.452413 & -0.458081 & 0.245985 & 0.225215 & 1.2529\%  \\

      \midrule
      \multirow{3}{*}{
    DBDP1~\cite{HCPHWX20}} 
                                          & 120  & -0.452413 & -0.730808 & 0.146088 & 0.278396 & 61.5357\%   \\
                                         & 130  & -0.452413 & -0.653904 & 0.207671 & 0.271590 & 44.5369\%  \\
                                         & 140 & -0.452413 & -0.703535 & 0.188429 & 0.271442 & 55.5073\% \\

      \bottomrule
    \end{tabular}}
  \end{threeparttable}
\end{table}

\begin{table}[!htb]
  \centering
  \begin{threeparttable}
    \caption{Comparison of numerical results for different schemes in Example~2 across different $N$ with $d=20$.}
    \label{tex2d20}
    \setlength{\tabcolsep}{1.3mm}{
   \begin{tabular}{lcccccc}
\toprule
Scheme
& \(N\)
& \(u(0,x_0)\)
& \(\overline{u}^{\,\pi}(0,x_0)\)
& Std. dev.
& \(\mathbb{E}[|u-u^{\pi}|]\)
& Rel. err.\\
\midrule
      \multirow{3}{*}{DBR}
                                       & 90  & 0.259041 & 0.249463 & 0.446021 & 0.366364 & 3.6974\%  \\
                                       & 100 & 0.259041 & 0.267776 & 0.417597& 0.305940 & 3.3720\%  \\
                                       & 140 & 0.259041 & 0.235651 & 0.264053 & 0.216113 & 9.0293\%  \\

      \midrule
      \multirow{2}{*}{
    DBDP1~\cite{HCPHWX20}}
                                         & 90  & 0.259041 & 0.076700 & 0.288634 & 0.264827 & 70.3907\%  \\
                                         & 100 & 0.259041 & 0.116744 & 0.230302 & 0.234747 & 54.9322\% \\
                                         & 140 & 0.259041 & 0.051147 & 0.286756 & 0.314811 & 80.2554\% \\

      \bottomrule
    \end{tabular}}
  \end{threeparttable}
\end{table}

However, a markedly different pattern emerges on this benchmark when the dimension
is increased to \(d=15\). As shown in Table~5.5, the DBDP1 method exhibits a
substantial deterioration in performance, with relative errors of approximately
\(61.5357\%\), \(44.5369\%\), and \(55.5073\%\) for \(N=120,130,140\), respectively.
These large errors indicate that DBDP1 is less reliable for this specific challenging
unbounded example at this dimension, rather than suggesting a universal failure of
DBDP1 for all problems with dimension larger than ten. In contrast, the DBR method
maintains relative errors below \(10\%\) across the tested temporal discretizations,
and achieves a relative error of \(1.2529\%\) at \(N=140\). This comparison highlights
the robustness of the DBR scheme for the present benchmark.

The performance gap becomes even more pronounced at \(d=20\) (Table~5.6). DBDP1
produces relative errors of \(70.3907\%\), \(54.9322\%\), and \(80.2554\%\) for
\(N=90,100,140\), respectively, whereas DBR achieves relative errors of
\(3.6974\%\), \(3.3720\%\), and \(9.0293\%\) under the same conditions. These results
show that, for this particular unbounded benchmark, the expectation-based DBR loss
leads to more stable and accurate estimates than the pathwise DBDP1 loss in the
tested high-dimensional settings.

These numerical results collectively support the core contribution of this work:
by replacing the pathwise loss function with expectation-based alternatives, the DBR
method can effectively reduce the adverse impact of pathwise noise and improve
training stability. For the challenging unbounded benchmark considered in Example~2,
the method's performance at \(d=15\) and \(d=20\), particularly when compared with
the substantial deterioration of DBDP1 on the same benchmark, provides evidence of
its practical utility for a class of high-dimensional nonlinear PDEs with complex
and unbounded structures. Due to prohibitive computational costs, the high-dimensional
regime \(d>20\) was excluded from our experiments. We anticipate that the results
would benefit from an increased sample size.

\section{Conclusion}

In this paper, we proposed a deep backward regression-based method for solving
high-dimensional quasilinear parabolic PDEs through their FBSDE representations.
The main idea is to replace pathwise Euler residuals by conditional-expectation-based
training targets. These targets are approximated through local multi-path Monte Carlo
averaging, which reduces the influence of Brownian pathwise noise and yields a
smoother regression problem for the neural networks.

Under the population-loss convention and exact minimization setting, we established
error estimates for the proposed DBR method. Under suitable approximation and
integrability assumptions, we further proved a half-order convergence result. We also
extended the method to variational inequalities and derived a corresponding error
estimate for the reflected case.

Several directions remain open. First, more efficient sampling strategies for
approximating the conditional expectations could reduce the computational cost of
the method. Second, lighter neural-network architectures may improve scalability in
very high dimensions. Finally, it would be valuable to incorporate finite-sample
generalization errors and optimization errors into the theoretical analysis, thereby
bridging the gap between the idealized population-loss theory and practical SGD
training.
\\
\\
\\
\textbf{ \Large Declarations}\\
\\
\textbf{Conflict of interest}\\
The authors declare that they have no conflict of interest.
\\
\\
\textbf{Data Availability}\\
All data used in the numerical experiments (Examples 1-2) are generated procedurally within the code using the parameters defined in the paper.
\\
\\
\textbf{ Authors' contributions }\\
The  authors contributed equally to this article. The authors read and approved the final manuscript.
\\
\\
\textbf{Funding}\\
This work was supported by the National Key R\&D Program of China (No.2023YFA1008701), the Key Project of the National Natural Science Foundation of China (No.12431017),
the National Natural Science of China (No.12501664),
the National Natural Science of Yangzhou, China (No.YZ2025145)
and the Golden Phoenix of the Green City-Yang Zhou (No.137013391).

\end{document}